\documentclass[
final
]{dmtcs-episciences}

\usepackage[utf8]{inputenc}
\usepackage{subfigure,amsfonts,amssymb,amsmath,amsthm,epsfig,euscript}
\usepackage[noadjust]{cite}
\usepackage{filecontents}

\newtheorem{theorem}{Theorem}
\newtheorem{lemma}[theorem]{Lemma}
\newtheorem{corollary}[theorem]{Corollary}

\newtheorem*{thm}{Theorem}

\long\def\symbolfootnote[#1]#2{\begingroup
\def\thefootnote{\fnsymbol{footnote}}\footnote[#1]{#2}\endgroup}

\newcommand{\la}{\lambda}
\newcommand{\La}{\Lambda}

\newcommand{\des}{\mathrm{des}}

\newcommand{\LRmin}{\mathrm{LRmin}}

\newcommand{\nth}[1][n]{{#1}^{\mathrm{th}}}

\newcommand{\qbinom}[2]{\genfrac{[}{]}{0pt}{}{#1}{#2}_{q}}
\newcommand{\pqbinom}[2]{\genfrac{[}{]}{0pt}{}{#1}{#2}_{p,q}}

\newcommand{\sg}{\sigma}

\newcommand{\cref}[1]{Corollary \ref{corollary:#1}}

\newcommand{\Floor}[1][n/2]{\left \lfloor #1 \right \rfloor}

\newcommand{\inv}{\mathrm{inv}}
\newcommand{\coinv}{\mathrm{coinv}}
\newcommand{\qbin}[3]{\genfrac{[}{]}{0pt}{}{#1}{#2}_{#3}}
\newcommand{\red}{\mathrm{red}}

\newcommand{\sgn}[1]{\mathrm{sgn}(#1)}

\newcommand{\tmch}{\text{$\tau$-$\mathrm{mch}$}}

\newcommand{\Gmch}{\text{$\Gamma$-$\mathrm{mch}$}}

\author{Quang T. Bach\affiliationmark{1}
  \and Jeffrey B. Remmel\affiliationmark{2} }  
\title{Descent c-Wilf Equivalence}
\affiliation{
   Department of Mathematics, UC San Diego, La Jolla, CA 92093-0112, U.S.A., qtbach@ucsd.edu \\
   Department of Mathematics, UC San Diego, La Jolla, CA 92093-0112, U.S.A.,remmel@math.ucsd.edu}
  
\received{2015-10-27}
\revised{2016-10-18}
\accepted{2017-2-8}

\begin{document}
\publicationdetails{18}{2017}{2}{13}{1312}
\maketitle
\begin{abstract}
  Let $S_n$ denote the symmetric group. For any $\sg \in S_n$, 
  we let $\des(\sg)$ denote the number of descents of $\sg$, 
  $\inv(\sg)$ denote the number of inversions of $\sg$, and 
  $\LRmin(\sg)$ denote the number of left-to-right minima of $\sg$.
  For any sequence of statistics $\mathrm{stat}_1, \ldots, \mathrm{stat}_k$ on permutations, we say two permutations $\alpha$ and $\beta$ in $S_j$ are 
  $(\mathrm{stat}_1, \ldots \mathrm{stat}_k)$-c-Wilf equivalent if 
  the generating function of $\prod_{i=1}^k x_i^{\mathrm{stat}_i}$ 
  over all permutations which have no consecutive occurrences of $\alpha$ 
  equals the generating function of $\prod_{i=1}^k x_i^{\mathrm{stat}_i}$ 
  over all permutations which have no consecutive occurrences of $\beta$. 
  We give many examples of pairs of permutations $\alpha$ and $\beta$ in $S_j$ 
  which are $\des$-c-Wilf equivalent, $(\des,\inv)$-c-Wilf equivalent, and 
  $(\des,\inv,\LRmin)$-c-Wilf equivalent.  For example, we will show 
  that if $\alpha$ and $\beta$ are minimally overlapping 
  permutations in $S_j$  which start with 1 and end with the same element and 
  $\des(\alpha) = \des(\beta)$ and $\inv(\alpha) = \inv(\beta)$, 
  then $\alpha$ and $\beta$ are $(\des,\inv)$-c-Wilf equivalent. 
\end{abstract}

\section{Introduction}

Let $S_n$ denote the group all permutations of $n$. That is, $S_n$ is the set of all one-to-one maps $\sg:\{1, \ldots, n\} \rightarrow \{1, \ldots, n\}$ under composition.  If $\sg = \sg_1 \ldots \sg_n \in S_n$, then we let $Des(\sg) = \{i: \sg_i >\sg_{i+1}\}$ and $\des(\sg) =|Des(\sg)|$. We say that $\sg_j$ is a \emph{left-to-right minima} of $\sg$ if $\sg_i > \sg_j$ for all $i < j$. For example the left-to-right minima  of $\sigma=938471625$ are 
$9$, $3$ and $1$.

Given a sequence $\tau = \tau_1 \cdots \tau_n$ of distinct positive integers, we define the \emph{reduction} of $\tau$, $\red(\tau)$, to be the permutation of $S_n$ that results by replacing the $i$-th smallest element of $\tau$ by $i$ for each $i$. For example $\red(53962) = 32541$. Let $\tau = \tau_1 \ldots \tau_j \in S_j$ and $\sg  = \sg_1 \ldots \sg_n \in S_n$.  Then we say that \begin{enumerate}
\item $\tau$ {\em occurs} in $\sg$ if there exists $1 \leq i_1 < \cdots < i_j \leq n$ such that $\red(\sg_{i_1}\sg_{i_2} \ldots \sg_{i_j}) = \tau$, 

\item there is a {\em $\tau$-match starting in position $i$} in $\sg$ if $\red(\sg_i \sg_{i+1} \ldots \sg_{i+j-1}) = \tau$, and 

\item $\sg$ {\em avoids} $\tau$ is there is no occurrence of $\tau$ in $\sg$.
\end{enumerate}
We let $\mathcal{S}_n(\tau)$ denote the set of permutations of $S_n$ which avoid $\tau$ and $\mathcal{NM}_n(\tau)$ denote the set of permutations of $S_n$ which have no $\tau$-matches. We let $S_n(\tau)= |\mathcal{S}_n(\tau)|$ and $NM_n(\tau) = |\mathcal{NM}_n(\tau)|$. If $\alpha$ and $\beta$ are elements of $S_j$, then we say that {\em $\alpha$ is Wilf equivalent to $\beta$} if $S_n(\alpha) = S_n(\beta)$ for all $n\geq 1$ and we say that {\em $\alpha$ is c-Wilf equivalent to $\beta$} if $NM_n(\alpha) = NM_n(\beta)$ for all $n$. For any permutations $\tau$ and $\sg$, we let $\tmch(\sg)$ denote the number of $\tau$-matches of $\sg$.

These definitions are easily extended to sets of permutations. That is, if $\Gamma \subseteq S_j$, then we say that \begin{enumerate}
\item $\Gamma$ {\em occurs} in $\sg$ if there exists $1 \leq i_1 < \cdots < i_j \leq n$ such that $\red(\sg_{i_1}\sg_{i_2} \ldots \sg_{i_j}) \in \Gamma$, 

\item there is a {\em $\Gamma$-match starting in position $i$} in $\sg$ if $\red(\sg_i \sg_{i+1} \ldots \sg_{i+j-1}) \in \Gamma$, and

\item $\sg$ {\em avoids} $\Gamma$ is there is no occurrence of $\Gamma$ in $\sg$.
\end{enumerate}
We let $\mathcal{S}_n(\Gamma)$ denote the set of permutations of $S_n$ which avoid $\Gamma$ and $\mathcal{NM}_n(\Gamma)$ denote the set of permutations of $S_n$ which have no $\Gamma$-matches. We let $S_n(\Gamma)= |\mathcal{S}_n(\Gamma)|$ and $NM_n(\Gamma) = |\mathcal{NM}_n(\Gamma)|$. If $\Gamma$ and $\Delta$ are subsets of $S_j$, then we say that $\Gamma$ is \emph{Wilf equivalent} to $\Delta$ if $S_n(\Gamma) = S_n(\Delta)$ for all $n$ and we say that $\Gamma$ is \emph{c-Wilf equivalent} to $\Delta$ if $NM_n(\Gamma) = NM_n(\Delta)$ for all $n$. For any permutation $\sg$ and set of permutations $\Gamma$, we let $\Gmch(\sg)$ denote the number of $\Gamma$-matches of $\sg$.

We let \begin{eqnarray*}
\  [n]_{p,q} &=& p^{n-1}+ p^{n-2}q + \cdots + pq^{n-2}+ q^{n-1} = 
\frac{p^n-q^n}{p-q},\\
\ [n]_{p,q}! &=& [1]_{p,q} [2]_{p,q} \cdots [n]_{p,q}, \ \mbox{and} \\
\ \pqbinom{n}{k} &=& \frac{[n]_{p,q} !}{[k]_{p,q}![n-k]_{p,q}!}
\end{eqnarray*}
denote the usual $p,q$-analogues of $n$, $n!$, and $\binom{n}{k}$. We shall use the standard conventions that $[0]_{p,q} = 0$ and $[0]_{p,q}! =1$. Setting $p=1$ in $[n]_{p,q}$, $[n]_{p,q}!$, and $\pqbinom{n}{k}$ yields $[n]_q$, $[n]_q!$, and $\qbinom{n}{k}$, respectively.

The main goal of this paper is to study refinements of the c-Wilf equivalence relation. For any permutation statistic $\mathbf{stat}$ on permutations and any pair of permutations $\alpha$ and $\beta$ in $S_j$, we say that $\alpha$ is $\mathbf{stat}$-c-Wilf equivalent to $\beta$ if for all $n \geq 1$
\begin{equation}
\sum_{\sg \in \mathcal{NM}_n(\alpha)} x^{\mathbf{stat}(\sg)}  = 
\sum_{\sg \in \mathcal{NM}_n(\beta)} x^{\mathbf{stat}(\sg)}.
\end{equation}
More generally, if $\mathbf{stat}_1, \ldots ,\mathbf{stat}_k$ are permutations statistics, then we say that two permutations $\alpha$ and $\beta$ are $(\mathbf{stat}_1, \ldots, \mathbf{stat}_k)$-c-Wilf equivalent if for all $n \geq 1$, 
\begin{equation*}
\sum_{\sg \in \mathcal{NM}_n(\alpha)} \prod_{i=1}^k x_i^{\mathbf{stat}_i(\sg)}  = \sum_{\sg \in \mathcal{NM}_n(\beta)} \prod_{i=1}^k x_i^{\mathbf{stat}_i(\sg)}.
\end{equation*}

The first question is whether there are interesting examples of $\mathbf{stat}$-c-Wilf equivalent permutations. The answer is yes. There are a number of such examples in the case where $\mathbf{stat}(\sg)$ is either $\inv(\sg)$, the number of inversions of $\sg$,  or $\coinv(\sg)$, the number of co-inversions of $\sg$. Here if $\sg = \sg_1 \ldots \sg_n \in S_n$, then 
\begin{eqnarray*}
\inv(\sg) &=& |\{(i,j): 1 \leq i < j \leq n \ \& \ \sg_i > \sg_j\}| 
\ \mbox{and} \\
\coinv(\sg) &=& |\{(i,j): 1 \leq i < j \leq n \ \& \ \sg_i < \sg_j\}|.
\end{eqnarray*}  
Since for any permutation $\sg \in S_n$, $\inv(\sg) + \coinv(\sg) = \binom{n}{2}$, it follows that 
\begin{equation}
\sum_{\sg \in \mathcal{NM}_n(\alpha)} x^{\mathrm{inv}(\sg)}  = \sum_{\sg \in \mathcal{NM}_n(\beta)} x^{\mathrm{inv}(\sg)}.
\end{equation}
if and only if 
\begin{equation}
\sum_{\sg \in \mathcal{NM}_n(\alpha)} x^{\mathrm{coinv}(\sg)}  = \sum_{\sg \in \mathcal{NM}_n(\beta)} x^{\mathrm{coinv}(\sg)}.
\end{equation}
Thus we will only consider $\inv$-c-Wilf equivalence.  It turns out that there are a large number of examples of $\alpha$ and $\beta$ which are $\inv$-c-Wilf equivalent when $\alpha$ and $\beta$ are minimal overlapping permutations.

We say that a permutation $\tau \in S_j$ where $j \geq 3$ has the {\em minimal overlapping property}, or is {\em minimal overlapping}, if the smallest $i$ such that there is a permutation $\sg \in  S_i$ with $\tmch(\sg) = 2$ is $2j-1$. This means that in any permutation $\sg =\sg_1 \ldots \sg_n$, any two $\tau$-matches in $\sg$ can share at most one letter which must be at the end of the first $\tau$-match and the start of the second $\tau$-match. For example, $\tau = 123$ does not have the minimal overlapping property since $\tmch(1234) =2$ and the $\tau$-match starting at position $1$ and the $\tau$-match starting at position 2 share two letters, namely, 2 and 3. However, it is easy to see that the permutation $\tau = 132$ does have the minimal overlapping property. That is, the fact that there is an ascent starting at position 1 and descent starting at position 2 means that there cannot be two $\tau$-matches in a permutation $\sg \in S_n$ which share two or more letters. If $\tau \in S_j$ has the minimal overlapping property, then  the shortest permutations $\sg$ such that $\tmch(\sg) =n$ will have length $n(j-1)+1$. Thus, we let $\mathcal{MP}_{\tau,n(j-1)+1}$ be the set of permutations $\sg \in S_{n(j-1)+1}$ such that $\tmch(\sg) =n$. We shall refer to the permutations in $\mathcal{MP}_{n,n(j-1)+1}$ as \emph{maximum packings} for $\tau$. Then we let $mp_{\tau,n(j-1)+1} =|\mathcal{MP}_{\tau,n(j-1)+1}|$ and 
$$mp_{\tau,n(j-1)+1}(p,q) =\sum_{\sg \in 
\mathcal{MP}_{\tau,n(j-1)+1}} q^{\inv(\sg)} p^{\coinv(\sg)}.$$

Duane and Remmel \cite{DR} proved the following theorem about minimal overlapping permutations. 
\begin{theorem}\label{thm:intro1}
If $\tau \in S_j$ has the minimal overlapping property, then 
\begin{multline}
\sum_{n \geq 0} \frac{t^n}{[n]_{p,q}!} \sum_{\sg \in S_n} x^{\tmch(\sg)}  
p^{\coinv(\sg)} q^{\inv(\sg)}
=  \\ \frac{1}{1 -(t+ \sum_{n \geq 1} \frac{t^{n(j-1)+1}}{[n(j-1)+1]_{p,q}!} 
(x-1)^{n} mp_{\tau,n(j-1)+1}(p,q))}. 
\end{multline}
\end{theorem}
They also proved the following theorem. 

\begin{theorem}\label{thm:intro2} Suppose that $\tau = \tau_1 \ldots \tau_j$ where $\tau_1 =1$ and $\tau_j = s$, then 
\begin{multline*}
mp_{\tau,(n+1)(j-1)+1}(p,q) = \\
p^{\coinv(\tau)}q^{inv(\tau)}p^{(s-1)n(j-1)}\pqbinom{(n+1)(j-1)+1 -s}{j-s} 
mp_{\tau,n(j-1)+1}(p,q)
\end{multline*}
so that 
\begin{equation}
mp_{\tau,(n+1)(j-1)+1}(p,q) = \left(p^{\coinv(\tau)}q^{inv(\tau)}\right)^{n+1}p^{(s-1)(j-1)\binom{n+1}{2}} \prod_{i=1}^{n+1} \pqbinom{i(j-1)+1-s}{j-s}.
\end{equation}
\end{theorem}

An immediate consequence of Theorems \ref{thm:intro1} and \ref{thm:intro2} is the following theorem. 
\begin{corollary}\label{cor:intro3}
Suppose that $\alpha= \alpha_1 \ldots \alpha_j$ and $\beta = \beta_1 \ldots \beta_j$ are permutations in $S_j$ such that $\alpha_1 = \beta_1 =1$, $\alpha_j = \beta_j =s$, $\inv(\alpha) = \inv(\beta)$, and $\alpha$ and $\beta$ have the minimal overlapping property.  Then 
\begin{equation}\label{eq:ab1}
\sum_{n \geq 0} \frac{t^n}{[n]_q!} \sum_{\sg \in S_n} x^{\alpha\mbox{-}\mathrm{mch}(\sg)} q^{\inv(\sg)} = 
\sum_{n \geq 0} \frac{t^n}{[n]_q!} \sum_{\sg \in S_n} x^{\beta\mbox{-}\mathrm{mch}(\sg)} q^{\inv(\sg)}.
\end{equation} 
\end{corollary}

We shall give several examples of a pair of permutations that $\alpha$ and $\beta$ that satisfy the hypotheses of Corollary \ref{cor:intro3} and thus are $\inv$-c-Wilf equivalent. It is easy to see that there are no pairs $\alpha$ and $\beta$ satisfying the hypothesis of Corollary \ref{cor:intro3} in $S_4$.  That is, there are only three possible pairs $\alpha$ and $\beta$ which start with 1 and end with the same numbers, namely, 
\begin{enumerate}
\item $\alpha = 1342$ and $\beta = 1432$, 
\item $\alpha = 1243$ and $\beta = 1423$, and 
\item $\alpha = 1234$ and $\beta = 1324$.
\end{enumerate}
In each case, $\inv(\alpha) \neq \inv(\beta)$. However $\alpha = 14532$ and $\beta = 15342$ do satisfy the hypothesis of Corollary \ref{cor:intro3}. Moreover it is easy to check that for any $n > 5$, any two permutations of the from $\alpha = 1453\sg2$ and $\beta = 1534\sg2$, where $\sg$ is the increasing sequence $678\cdots n$, satisfy the hypothesis of Corollary \ref{cor:intro3}. Thus, there are non-trivial examples of $\inv$-c-Wilf equivalence for all $n \geq 1$. In fact, Duane and Remmel proved an even stronger result than Theorem \ref{thm:intro2}. That is, they proved the following theorem.  

\begin{theorem}\label{thm:abmatch} 
Suppose $\alpha = \alpha_1 \ldots \alpha_j$ and $\beta = \beta_1 \ldots \beta_j$ are minimal overlapping permutations in $S_j$ and $\alpha_1 = \beta_1$ and $\alpha_j = \beta_j$, then for all $n \geq 1$, 
\begin{equation}\label{eq:abmatch}
mp_{\alpha,n(j-1)+1} =mp_{\beta,n(j-1)+1}.
\end{equation}
If in addition, $p^{\coinv(\alpha)} q^{\inv(\alpha)} = p^{\coinv{(\beta)}} q^{\inv{(\beta)}}$, then 
\begin{equation}\label{eq:abmatchpq}
mp_{\alpha,n(j-1)+1}(p,q) =mp_{\beta,n(j-1)+1}(p,q).
\end{equation}
\end{theorem}

Combining Theorems \ref{thm:intro1} and \ref{thm:abmatch}, we have the following theorem.  
\begin{theorem}\label{thm:intro3}
Suppose that $\alpha= \alpha_1 \ldots \alpha_j$ and $\beta = \beta_1 \ldots \beta_j$ are permutations in $S_j$ such that $\alpha_1 = \beta_1$, $\alpha_j = \beta_j$, $\inv(\alpha) = \inv(\beta)$, and $\alpha$ and $\beta$ have the minimal overlapping property.  Then 
\begin{equation}\label{eq:ab}
\sum_{n \geq 0} \frac{t^n}{[n]_q!} \sum_{\sg \in S_n} x^{\alpha\mbox{-}\mathrm{mch}(\sg)} q^{\inv(\sg)} = 
\sum_{n \geq 0} \frac{t^n}{[n]_q!} \sum_{\sg \in S_n} x^{\beta\mbox{-}\mathrm{mch}(\sg)} q^{\inv(\sg)}.
\end{equation} 
\end{theorem}
Theorem \ref{thm:intro3} above relaxes the condition that $\alpha$ and $\beta$ both have to start with 1 and thus, introduces a stronger condition than just being $\inv$-c-Wilf equivalent. In fact, we shall say that $\alpha$ and $\beta$ are {\em strongly $\inv$-c-Wilf equivalent} if they satisfy the hypotheses of Theorem \ref{thm:intro3}. As an example, one can check that $\alpha = 241365$ and $\beta = 234165$ both start and end with the same element and have the same number of inversions. Therefore, they are strongly $\inv$-c-Wilf equivalent.

Of course, one can make similar definitions in the case where we replace $c$-Wilf equivalence by Wilf equivalence. For example, we say that $\alpha$ is $\mathbf{stat}$-Wilf equivalent to $\beta$ if for all $n \geq 1$
\begin{equation}
\sum_{\sg \in S_n(\alpha)} x^{\mathbf{stat}(\sg)}  = 
\sum_{\sg \in S_n(\beta)} x^{\mathbf{stat}(\sg)}.
\end{equation}
Although this language has not been used, there are numerous examples in the literature where researchers have given a bijection $\phi_n:S_n(\alpha) \rightarrow S_n(\beta)$ to prove that $\alpha$ and $\beta$ are Wilf equivalent where the bijection $\phi_n$ preserves other statistics. One example of this phenomenon is the work of Claesson and Kitaev \cite{CK} who gave a classification of various bijections between 321-avoiding and 132-avoiding permutations according to what statistics they preserved.

The main goal of this paper is to give examples of $\alpha$ and $\beta$ such that $\alpha$ and $\beta$ are $\des$-c-Wilf equivalent. Our main result is the following. 

\begin{theorem}\label{thm:intro4}
Suppose that $\alpha= \alpha_1 \ldots \alpha_j$ and $\beta = \beta_1 \ldots \beta_j$ are permutations in $S_j$ such that $\alpha_1 = \beta_1=1$, $\alpha_j = \beta_j$, $\des(\alpha) = \des(\beta)$, and $\alpha$ and $\beta$ have the minimal overlapping property.  Then 
\begin{equation}\label{eq:abdes1}
\sum_{n \geq 0} \frac{t^n}{n!} \sum_{\sg \in \mathcal{NM}_n(\alpha)}  x^{\des(\sg)} = 
\sum_{n \geq 0} \frac{t^n}{n!} \sum_{\sg \in \mathcal{NM}_n(\beta)} x^{\des(\sg)}.
\end{equation} 
Thus $\alpha$ and $\beta$ are $\des$-c-Wilf equivalent.

If in addition, $\inv(\alpha) = \inv(\beta)$, then 
\begin{equation}\label{eq:abdesq1}
\sum_{n \geq 0} \frac{t^n}{[n]_q!} \sum_{\sg \in \mathcal{NM}_n(\alpha)}  
x^{\des(\sg)} q^{\inv(\sg)}= 
\sum_{n \geq 0} \frac{t^n}{[n]_q!} \sum_{\sg \in \mathcal{NM}_n(\beta)} 
x^{\des(\sg)} q^{\inv(\sg)}.
\end{equation} 
Thus $\alpha$ and $\beta$ are $(\des,\inv)$-c-Wilf equivalent.
\end{theorem}

To prove Theorem \ref{thm:intro4}, we will modify the reciprocity method of Jones and Remmel \cite{JR,JR1,JR2}. The reciprocity method was designed to study generating functions of the form 
\begin{equation}\label{eq:I1}
\mbox{NM}_{\Gamma}(t,x,y)=\sum_{n \geq 0} \frac{t^n}{n!} \mbox{NM}_{\Gamma,n}(x,y)
\end{equation}
where $\displaystyle \mbox{NM}_{\Gamma,n}(x,y) =\sum_{\sg \in \mathcal{NM}_n(\Gamma)}x^{\LRmin(\sg)}y^{1+\des(\sg)}$. In the special case where $\Gamma = \{\tau\}$ is a set with a single permutation $\tau$, we shall write $\mbox{NM}_{\tau}(t,x,y)$ for $\mbox{NM}_{\Gamma}(t,x,y)$ and $\mbox{NM}_{\tau,n}(x,y)$ for $\mbox{NM}_{\Gamma,n}(x,y)$.

The basic idea of their approach to study the generating functions $\mbox{NM}_{\tau}(t,x,y)$ is as follows. First, assume that $\tau$ starts with 1 and $\des(\tau) =1$. It follows from results in \cite{JR} that if $\tau$ starts with 1, then we can write $\mbox{NM}_{\tau}(t,x,y)$ in the form 
\begin{equation}\label{eq:I2}
\mbox{NM}_{\tau}(t,x,y) = \left( \frac{1}{U_{\tau}(t,y)}\right)^x
\end{equation}
where $\displaystyle U_{\tau}(t,y) = \sum_{n\geq 0}U_{\tau,n}(y) \frac{t^n}{n!}$.

Next one writes 
\begin{equation}\label{eq:I3}
U_{\tau}(t,y) = \frac{1}{1+\sum_{n \geq 1} \mbox{NM}_{\tau,n}(1,y) \frac{t^n}{n!}}.
\end{equation} One can then use the homomorphism method to give a combinatorial interpretation of the right-hand side of (\ref{eq:I3}) which can be used to find simple recursions for the coefficients $U_{\tau,n}(y)$. The homomorphism method derives generating functions for various permutation statistics by applying a ring homomorphism defined on the ring of symmetric functions \begin{math}\Lambda\end{math} in infinitely many variables \begin{math}x_1,x_2, \ldots \end{math} to simple symmetric function identities such as \begin{equation}\label{conclusion2}
H(t) = 1/E(-t),
\end{equation}
where $H(t)$ and $E(t)$ are the generating functions for the homogeneous and elementary symmetric functions, respectively:
\begin{equation}\label{genfns}
H(t) = \sum_{n\geq 0} h_n t^n = \prod_{i\geq 1} \frac{1}{1-x_it} \ \mbox{and} \ E(t) = \sum_{n\geq 0} e_n t^n = \prod_{i\geq 1} (1+x_it).
\end{equation}
In their case, Jones and Remmel defined a homomorphism \begin{math}\theta_{\tau}\end{math} on \begin{math}\Lambda\end{math} by setting \begin{displaymath}\theta_{\tau}(e_n) = \frac{(-1)^n}{n!} \mbox{NM}_{\tau,n}(1,y).\end{displaymath}
Then \begin{displaymath}\theta_{\tau}(E(-t)) = {\sum_{n\geq 0} \mbox{NM}_{\tau,n}(1,y) \frac{t^n}{n!}} = \frac{1}{U_\tau(t,y)}.\end{displaymath}
Hence 
$$U_\tau(t,y) = \frac{1}{\theta_{\tau}(E(-t))} = \theta_{\tau}(H(t)),$$
which implies that 
\begin{equation}\label{eq:combhn}
n!\theta_{\tau}(h_n) = U_{\tau,n}(y).
\end{equation}
Thus, if we can compute $n!\theta_{\tau}(h_n)$ for all $n \geq 1$, then we can compute the polynomials $U_{\tau,n}(y)$ and the generating function $U_{\tau}(t,y)$, which in turn allows us to compute the generating function $\mbox{NM}_{\tau}(t,x,y)$. Jones and Remmel \cite{JR2,JR3} showed that one can interpret $n!\theta_{\tau}(h_n)$ as a certain signed sum of the weights of filled, labeled brick tabloids when $\tau$ starts with 1 and $\des(\tau)=1$. They then defined a weight-preserving, sign-reversing involution $I$ on the set of such filled, labeled brick tabloids which allowed them to give a relatively simple combinatorial interpretation for $n!\theta_{\tau}(n_n)$. Then they showed how such a combinatorial interpretation allowed them to prove that for certain families of such permutations $\tau$, the polynomials $U_{\tau,n}(y)$ satisfy simple recursions. 

For example, in \cite{JR2}, Jones and Remmel studied the generating functions $\mbox{NM}_{\tau}(t,x,y)$ for permutations $\tau$ of the form $\tau = 1324\cdots p$ where $p \geq 4$.  Using the reciprocity method, they proved that $U_{1324,1}(y)=-y$ and for $n \geq 2$, 
\begin{equation}\label{1324}
U_{1324,n}(y) = (1-y)U_{1324,n-1}(y)+ \sum_{k=2}^{\lfloor n/2 \rfloor} (-y)^{k-1} C_{k-1} U_{1324,n-2k+1}(y),
\end{equation} where $C_k = \frac{1}{k+1}\binom{2k}{k}$ is the $k^{th}$ Catalan number. They also proved that for  any $p \geq 5$,  $U_{1324 \cdots p,n}(y) =-y$ and for $n \geq 2$, 
\begin{equation}\label{1324p}
U_{1324\cdots p,n}(y)=(1-y)U_{1324\cdots p,n-1}(y)+\sum_{k=2}^{\lfloor\frac{n-2}{p-2}\rfloor+1}(-y)^{k-1}U_{1324\cdots p,n-((k-1)(p-2)+1)}(y).
\end{equation}

In \cite{BR}, the authors  extended the reciprocity method of Jones and Remmel to study the polynomials $U_{\Gamma,n}(y)$ where 
\begin{equation*} 
U_{\Gamma}(t,y) = 1+ \sum_{n \geq 1} U_{\Gamma,n}(y) \frac{t^n}{n!} = \frac{1}{1+\sum_{n \geq 1} \mbox{NM}_{\Gamma,n}(1,y) \frac{t^n}{n!}}
\end{equation*}
in the case where $\Gamma$ is a set of permutations such that for all $\tau \in \Gamma$, $\tau$ starts with 1 and $\des(\tau) \leq 1$. For example, suppose that $k_1, k_2 \geq 2$, $p = k_1 + k_2$, and  
$$\Gamma_{k_1,k_2} = \{\sigma \in S_p: \sigma_1=1, \sigma_{k_1+1}=2, \sigma_1 < \sigma_2< \cdots<\sigma_{k_1}~ \&~\sigma_{k_1+1} < \sigma_{k_1+2}< \cdots<\sigma_{p} \}.$$ 
That is, $\Gamma_{k_1,k_2}$ consists of all permutations $\sg$ of length $p$ where 1 is in position 1, 2 is in position $k_1+1$, and $\sg$ consists of two increasing sequences, one starting at 1 and the other starting at 2. In \cite{BR}, we proved that for $\Gamma = \Gamma_{k_1,k_2}$, $U_{\Gamma,1}(y)=-y$, and for $n \geq 2,$
\begin{align*}
\displaystyle   U_{\Gamma,n}(y) &= (1-y)U_{\Gamma,n-1}(y) -y\binom{n-2}{k_1-1}\left( U_{\Gamma,n-M}(y) +y\sum_{i=1}^{m-1}U_{\Gamma,n-M-i}(y) \right), 
\end{align*} where $m = \min\{k_1, k_2\}$, and $M = \max\{k_1,k_2\}$. 

Furthermore, in \cite{BR}, we investigated a new phenomenon that arises when we add the identity permutation $12 \ldots k$ to the family $\Gamma$. For example, if $\Gamma = \{1324,123\}$, then we proved that  $U_{\Gamma,1}(y)=-y$, and for $n \geq 2,$
\begin{equation}\label{1324-123}
U_{\Gamma,n}(y) = -yU_{\Gamma,n-1}(y) -yU_{\Gamma,n-2}(y) + \sum_{k=2}^{\Floor }(-y)^{k}C_{k-1}U_{\Gamma, n-2k}(y),
\end{equation} and when $\Gamma = \{1324\dots p,123\dots p-1\}$ where $p \geq 5,$ then $U_{\Gamma,1}(y)=-y$, and for $n \geq 2,$
\begin{equation}\label{1324p-12p}
U_{\Gamma,n}(y) = \sum_{k=1}^{p-2}(-y)U_{\Gamma,n-k}(y) + \sum_{k=1}^{p-2}\sum_{m=2}^{\lfloor\frac{n-k}{p-2}\rfloor }(-y)^{m}U_{\Gamma, n-k-(m-1)(p-2)}(y). 
\end{equation} 
While on the surface, the recursions (\ref{1324-123}) and (\ref{1324p-12p}) do not seem to be simpler than the corresponding recursions (\ref{1324}) and (\ref{1324p}), they are easier to analyze because adding an identity permutation $12 \ldots k$ to $\Gamma$ ensures that all the bricks in the filled brick tabloids used to interpret $n!\theta_{\tau}(h_n)$ have length less than $k$. For example, in \cite{BR}, we were able to prove the following explicit formula for the polynomials $U_{\{1324,123\},n}(y)$. 

\begin{theorem} \label{1324,123}
Let $\Gamma = \{1324,123\}$. Then for all $n \geq 0$, 
\begin{equation}\label{u2n}
U_{\Gamma,2n}(y) = \sum_{k=0}^n \frac{(2k+1)\binom{2n}{n-k}}{n+k+1}(-y)^{n+k+1}
\end{equation} and 
\begin{equation}\label{u2n+1}
U_{\Gamma,2n+1}(y) = \sum_{k=0}^n \frac{2(k+1)\binom{2n+1}{n-k}}{n+k+2}(-y)^{n+k}.
\end{equation}
\end{theorem}

Another example in \cite{BR} where we could find an explicit formula is the following. Let $\Gamma_{k_1,k_2,s} = \Gamma_{k_1,k_2} \cup \{1 \cdots s(s+1)\}$ for some $s \geq \max(k_1,k_2)$. In \cite{BR}, we showed that $U_{\Gamma_{2,2,s},1}(y)=-y$, and for $n \geq 2,$
\begin{align} \label{Gamma22s}
U_{\Gamma_{2,2,s},n}(y)&= -yU_{\Gamma_{2,2,s},n-1}(y) - \nonumber\\
&\ \ \ \sum_{k=0}^{s-2} \left((n-k-1) yU_{\Gamma_{2,2,s},n-k-2}(y)+(n-k-2)y^2 U_{\Gamma_{2,2,s},n-k-3}(y)\right).
\end{align} 
Using these recursions, we then proved that 
\begin{align*}
U_{\Gamma_{2,2,2},2n}(y) = & \sum_{i=0}^n (2n-1)\downarrow \downarrow_{n-i} (-y)^{n+i} \ \mbox{~~and} \\
U_{\Gamma_{2,2,2},2n+1}(y) =& \sum_{i=0}^n (2n) \downarrow \downarrow_{n-i} (-y)^{n+1+i},
\end{align*} where for any $x$, $(x)\downarrow \downarrow_{0} =1$ and $(x)\downarrow \downarrow_{k} =x(x-2)(x-4) \cdots (x-2k -2)$ for $k \geq 1$.

In a subsequent paper \cite{BR2}, the authors further extended the reciprocity method to study the generating functions $\mbox{NM}_{\Gamma}(t,x,y)$ where all the permutations $\Gamma$ start with 1 but where we do not put any condition on the number of descents in a permutations in $\Gamma$. It is this extension that we shall modify to prove Theorem \ref{thm:intro4}. In particular, we shall be interested in computing generating functions of the form 
$$\mbox{INM}_{\Gamma}(t,q,z) =
1+ \sum_{n \geq 0} \frac{t^n}{[n]_q!}  \mbox{INM}_{\Gamma,n}(q,z), $$
where $\displaystyle \mbox{INM}_{\Gamma,n}(q,z)= \sum_{\sg \in \mathcal{NM}_n(\Gamma)} z^{\des(\sg)+1} q^{\inv(\sg)}$, 
which is a $q$-analogue of $\mbox{NM}_{\Gamma}(t,1,y)$. 
We shall write 
$$\mbox{INM}_{\Gamma}(t,q,z) = \frac{1}{1+ \sum_{n \geq 1} \mbox{IU}_{\Gamma,n}(q,z) \frac{t^n}{[n]_q!}}
$$ 
so that \begin{equation}\label{introkey}
\mbox{IU}(t,q,z)= 1+ \sum_{n \geq 1} \mbox{IU}_{\Gamma,n}(q,z) \frac{t^n}{[n]_q!} = \frac{1}{\mbox{INM}_{\Gamma}(t,q,z)}.
\end{equation}
Again if $\Gamma =\{\tau\}$, we shall write $\mbox{INM}_{\tau}(t,q,z)$ for $\mbox{INM}_{\Gamma}(t,q,z)$, $\mbox{INM}_{\tau,n}(q,z)$ for $\mbox{INM}_{\Gamma,n}(q,z)$, $\mbox{IU}_{\tau}(t,q,z)$ for $\mbox{IU}_{\Gamma}(t,q,z)$, and $\mbox{IU}_{\tau,n}(q,z)$ for $\mbox{IU}_{\Gamma,n}(q,z)$. As before, we shall use the homomorphism method to give us a combinatorial interpretation of the right-hand side of (\ref{introkey}) which can be used to develop recursions for $\mbox{IU}_{\Gamma,n}(q,z)$. In the case where $\alpha$ and $\beta$ satisfy all the hypothesis of Theorem \ref{thm:intro4}, then we will show that $\mbox{IU}_{\alpha,n}(q,z)$ and  $\mbox{IU}_{\beta,n}(q,z)$ satisfy the same recursions so that $\mbox{INM}_{\alpha}(t,q,z) = \mbox{INM}_{\beta}(t,q,z)$.

Finally, there are stronger conditions on permutations $\alpha$ and $\beta$ in $S_j$ which will guarantee that $\alpha$ and $\beta$ are $\des$-c-Wilf equivalent, $(\des,\inv)$-c-Wilf equivalent, or $(\des,\inv,\LRmin)$-c-Wilf equivalent. That is, we say that $\alpha$ and $\beta$ are {\em mutually minimal overlapping} if $\alpha$ and $\beta$ are minimal overlapping and the smallest $n$ such that there exist a permutation $\sg \in S_n$ with $\alpha\mbox{-}\mathrm{mch}(\sg) \geq 1$ and $\beta\mbox{-}\mathrm{mch}(\sg) \geq 1$ is $2j-1$. This ensures that in any permutation $\sg$, any pair of $\alpha$-matches, any pair of $\beta$ matches, and any pair of matches where one match is an $\alpha$-match and one match is a $\beta$-match can share at most one letter.  There are lots of examples of minimal overlapping permutations $\alpha$ and $\beta$ in $S_j$ such that $\alpha$ and $\beta$ are mutually minimal overlapping. For example, we shall prove that any minimal overlapping pair of permutations $\alpha$ and $\beta$ in $S_j$ which start with 1 and end with 2 are automatically mutually minimal overlapping. We will also give examples of minimal overlapping permutations $\alpha =\alpha_1 \ldots \alpha_j$ and $\beta = \beta_1 \ldots \beta_j$ in $S_j$ such that $\alpha_1 = \beta_1 =1$ and $\alpha_j = \beta_j$ which are not mutually minimal overlapping. Then we shall give a bijective proof the following theorem.

\begin{theorem}\label{thm:intro5}
Suppose $\alpha= \alpha_1 \ldots \alpha_j$ and $\beta = \beta_1 \ldots \beta_j$ are permutations in $S_j$ which are mutually minimal overlapping and there is an $1 \leq a < j$ such that $\alpha_i = \beta_i$ for all $i \leq a$, $\alpha_a = \beta_a = 1$, $\alpha_j = \beta_j$, and $\des(\alpha) = \des(\beta)$.  Then 
\begin{equation}\label{eq:ABdesRLmin}
\sum_{n \geq 0} \frac{t^n}{n!} \sum_{\sg \in \mathcal{NM}_n(\alpha)}  z^{\des(\sg)} u^{\LRmin(\sg)}= 
\sum_{n \geq 0} \frac{t^n}{n!} \sum_{\sg \in \mathcal{NM}_n(\beta)} z^{\des(\sg)}u^{\LRmin(\sg)}.
\end{equation} 
Thus $\alpha$ and $\beta$ are $(\des,\LRmin)$-c-Wilf equivalent.

If in addition, $\inv(\alpha) = \inv(\beta)$, then 
\begin{equation}\label{eq:ABdesq}
\sum_{n \geq 0} \frac{t^n}{[n]_q!} \sum_{\sg \in \mathcal{NM}_n(\alpha)}  x^{\des(\sg)} q^{\inv(\sg)} u^{\LRmin(\sg)}= 
\sum_{n \geq 0} \frac{t^n}{n!} \sum_{\sg \in \mathcal{NM}_n(\beta)} x^{\des(\sg)} q^{\inv(\sg)}u^{\LRmin(\sg)}.
\end{equation} 
Thus $\alpha$ and $\beta$ are $(\des,\inv,\LRmin)$-c-Wilf equivalent.
\end{theorem}

The outline of this paper is the following. In section 2, we shall give the background in symmetric functions needed for our proofs. In Section 3, we shall modify the involution defined in \cite{BR2} to give a combinatorial interpretation to give $\mbox{IU}_{\Gamma,n}(q,z)$.  In section 4, we shall use the results of Section 3 to prove Theorem \ref{thm:intro4}  and give several examples of families of pairs of permutations satisfying the hypothesis of Theorem \ref{thm:intro4}. Finally in section 5, we shall prove a stronger result than Theorem \ref{thm:intro5} which will immediately imply Theorem \ref{thm:intro5} and give several examples of families of pairs of permutations satisfying the hypothesis of Theorem \ref{thm:intro5}.

\section{Symmetric Functions}

In this section, we give the necessary background on symmetric functions that will be used in our proofs. 

A partition of $n$ is a sequence of positive integers \begin{math}\la = (\la_1, \ldots ,\la_s)\end{math} such that \begin{math}0 < \la_1 \leq \cdots \leq \la_s\end{math} and $n=\la_1+ \cdots +\la_s$. We shall write $\lambda \vdash n$ to denote that $\lambda$ is partition of $n$ and we let $\ell(\lambda)$ denote the number of parts of $\lambda$. When a partition of $n$ involves repeated parts, we shall often use exponents in the partition notation to indicate these repeated parts. For example, we will write $(1^2,4^5)$ for the partition $(1,1,4,4,4,4,4)$.

Let \begin{math}\Lambda\end{math} denote the ring of symmetric functions in infinitely  many variables \begin{math}x_1,x_2, \ldots \end{math}. The \begin{math}\nth\end{math} elementary symmetric function \begin{math}e_n = e_n(x_1,x_2, \ldots )\end{math}  and \begin{math}\nth\end{math} homogeneous symmetric function \begin{math}h_n = h_n(x_1,x_2, \ldots )\end{math} are defined by the generating functions given in (\ref{genfns}). For any partition \begin{math}\la = (\la_1,\dots,\la_\ell)\end{math}, let \begin{math}e_\la = e_{\la_1} \cdots e_{\la_\ell}\end{math} and \begin{math}h_\la = h_{\la_1} \cdots h_{\la_\ell}\end{math}.  It is well known that \begin{math}e_0,e_1, \ldots \end{math} is an algebraically independent set of generators for \begin{math}\La\end{math}, and hence, a ring homomorphism \begin{math}\theta\end{math} on \begin{math}\Lambda\end{math} can be defined  by simply specifying \begin{math}\theta(e_n)\end{math} for all \begin{math}n\end{math}.

If $\lambda =(\lambda_1, \ldots, \lambda_k)$ is a partition of $n$, then a $\lambda$-brick tabloid of shape $(n)$ is a filling of a rectangle consisting of $n$ cells with bricks of sizes $\lambda_1, \ldots, \lambda_k$ in such a way that no two bricks overlap. For example, Figure \ref{fig:DIMfig1} shows the six $(1^2,2^2)$-brick tabloids of shape $(6)$. 

\begin{figure}[htbp]
  \begin{center}
    \includegraphics[width=0.7\textwidth]{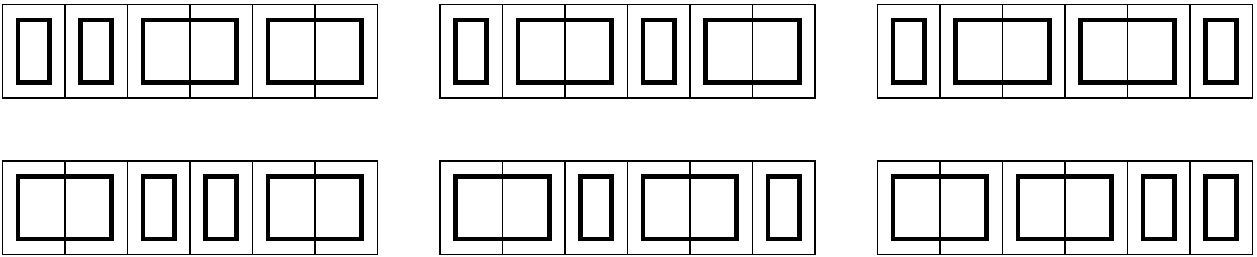}
    \caption{The six $(1^2,2^2)$-brick tabloids of shape $(6)$.}
    \label{fig:DIMfig1}
  \end{center}
\end{figure}


Let \begin{math}\mathcal{B}_{\la,n}\end{math} denote the set of \begin{math}\la\end{math}-brick tabloids of shape \begin{math}(n)\end{math} and let \begin{math}B_{\la,n}\end{math} be the number of \begin{math}\la\end{math}-brick tabloids of shape \begin{math}(n)\end{math}.  If \begin{math}B \in \mathcal{B}_{\la,n}\end{math}, we will write \begin{math}B =(b_1, \ldots, b_{\ell(\la)})\end{math} if the lengths of the bricks in \begin{math}B\end{math}, reading from left to right, are \begin{math}b_1, \ldots, b_{\ell(\la)}\end{math}. For example, the brick tabloid in the top right position in Figure \ref{fig:DIMfig1} is denoted as $(1,2,2,1)$. E\u{g}ecio\u{g}lu and the second author  \cite{Eg1} proved that 
\begin{equation}\label{htoe}
h_n = \sum_{\la \vdash n} (-1)^{n - \ell(\la)} B_{\la,n}~ e_\la.
\end{equation}
This interpretation of $h_n$ in terms of $e_n$ will aid us in describing the coefficients of $\theta_\Gamma(H(t))=\mbox{IU}_\Gamma(t,q,z)$ coming in the next section, which will in turn allow us to compute the coefficients 
$\mbox{INM}_{\Gamma,n}(q,z)$.

\section{A $q$-analogue of the reciprocity method}
In the section, we shall modify the results of \cite{BR2} to give a combinatorial interpretation of $IU_{\Gamma,n}(q,z)$ in the case where all the permutations in $\Gamma$ start with 1.

Let \begin{equation}
\mbox{INM}_{\Gamma,n}(q,z) = \sum_{\sg \in \mathcal{NM}_n(\Gamma)} z^{\des(\sg)+1} q^{\inv(\sg)}.
\end{equation}
We define a ring homomorphism $\theta_{\Gamma}$ on the ring of symmetric functions $\Lambda$ by setting $\theta_{\Gamma}(e_0) = 1$ and
\begin{equation} \label{def:Theta} 
\theta_{\Gamma}(e_n) = \frac{(-1)^n}{[n]_q!} \mbox{INM}_{\Gamma,n}(q,z)
\end{equation} for $n \geq 1$.  
It then follows that 
\begin{eqnarray} \label{theta=u}
\theta_{\Gamma}(H(t)) &=& \sum_{n \geq 0} \theta_{\Gamma}(h_n)t^n  = \frac{1}{\theta_{\tau}(E(-t))} = \frac{1}{1 + \sum_{n \geq 1} (-t)^n \theta_{\Gamma}(e_n)} \nonumber \\
&=& \frac{1}{1 + \sum_{n \geq 1} \frac{t^n}{[n]_q!} \mbox{INM}_{\Gamma,n}(q,z)} = \mbox{IU}_{\Gamma}(t,q,z).
\end{eqnarray}

Using (\ref{htoe}), we can compute \begin{eqnarray}\label{eq:basic1}
[n]_q! \theta_{\Gamma}(h_n) &=& [n]_q! \sum_{\la \vdash n} (-1)^{n-\ell(\la)} B_{\la,n}~ \theta_{\Gamma}(e_\la) \nonumber \\
&=& [n]_q! \sum_{\la \vdash n} (-1)^{n-\ell(\la)} \sum_{(b_1, \ldots, b_{\ell(\la)}) \in \mathcal{B}_{\la,n}} \prod_{i=1}^{\ell(\la)}  \frac{(-1)^{b_i}}{[b_i]_q!} \mbox{INM}_{\Gamma,b_i}(q,z) \nonumber \\
&=& \sum_{\la \vdash n} (-1)^{\ell(\la)} \sum_{(b_1, \ldots, b_{\ell(\la)}) \in \mathcal{B}_{\la,n}} \qbin{n}{b_1, \ldots, b_{\ell(\la)}}{q}\prod_{i=1}^{\ell(\la)}  \mbox{INM}_{\Gamma,b_i}(q,z).
\end{eqnarray}

To give combinatorial interpretation to the right hand side of (\ref{eq:basic1}), we select a brick tabloid $B = (b_1, b_2, \dots, b_{\ell(\la)} )$ of shape $(n)$ filled with bricks whose sizes induce the partition $\la$. Given an ordered set partition $\mathcal{S} =(S_1, S_2, \dots, S_{\ell(\la)})$ of $\{1,2, \dots, n\}$ such that $|S_i| = b_i,$ for $i = 1, \dots, \ell(\la)$, let $S_1\uparrow S_2\uparrow \ldots S_{\ell(\la)}\uparrow$ denote the permutation of $S_n$ which results by taking the elements of $S_i$ in increasing order and concatenating them from left to right. For example, 
$$\{1,5,6\}\uparrow \{7,9\}\uparrow \{2,3,4,8\}\uparrow = 156792348.$$
It follows from results in \cite{BeckRem} that we can interpret the $q$-multinomial coefficient $\qbin{n}{b_1, \ldots, b_{\ell(\la)}}{q}$ as the sum of $q^{\inv(S_1\uparrow S_2\uparrow \ldots S_{\ell(\la)}\uparrow)}$ over all ordered set partitions $\mathcal{S} =(S_1, S_2, \dots, S_{\ell(\la)})$ of $\{1,2, \dots, n\}$ such that $|S_i| = b_i,$ for $i = 1, \dots, \ell(\la).$ For each brick $b_i,$ we then fill the cells of $b_i$ with numbers from $S_i$ such that the entries in the brick reduce to a permutation $\sg^{(i)} = \sg_1 \cdots \sg_{b_i}$ in $\mathcal{NM}_{b_i}(\Gamma)$. It follows that if we sum $q^{\inv(\sg)}$ over all possible choices of $(S_1, S_2, \dots, S_{\ell(\la)})$, we will obtain 
$$\qbin{n}{b_1, \ldots, b_{\ell(\la)}}{q}\prod_{i=1}^{\ell(\mu)} q^{\inv(\sg^{(i)})}.$$
We label each descent of $\sg$ that occurs within each brick as well as the last cell of each brick by $z$.  This accounts for the factor $z^{\des(\sg^{(i)})+1}$ within each brick. Finally, we use the factor $(-1)^{\ell(\la)}$ to change the label of the last cell of each brick from $z$ to $-z$. We will denote the filled labeled brick tabloid constructed in this way as $\langle B,\mathcal{S},(\sg^{(1)}, \ldots, \sg^{(\ell(\la))})\rangle$. 

For example, when $n = 17, \Gamma = \{1324, 1423, 12345\},$ and $B = (9,3,5,2),$ consider the ordered set partition $\mathcal{S}=(S_1,S_2,S_3,S_4)$ of $\{1,2,\dots, 19\}$ where $S_1=\{2,5,6,9,11,15,16,17,19\},$ $S_2 = \{7,8,14\},$ $S_3 = \{1,3,10,13,18\},$ $S_4 = \{4,12\}$ and the permutations $\sg^{(1)} = 1~2~4~6~5~3~7~9~8 \in \mathcal{NM}_{9}(\Gamma),$ $\sg^{(2)} = 1~3~2 \in \mathcal{NM}_{3}(\Gamma),$ $\sg^{(3)} = 5~1~2~4~3 \in \mathcal{NM}_{5}(\Gamma),$ and $\sg^{(4)} = 2~1 \in \mathcal{NM}_{2}(\Gamma)$.  Then the construction of $\langle B,\mathcal{S},(\sg^{(1)}, \ldots, \sg^{(4)})\rangle$ is pictured in Figure \ref{fig:CFLBTab}.

\begin{figure}[htbp]
  \begin{center}
    \includegraphics[width=0.7\textwidth]{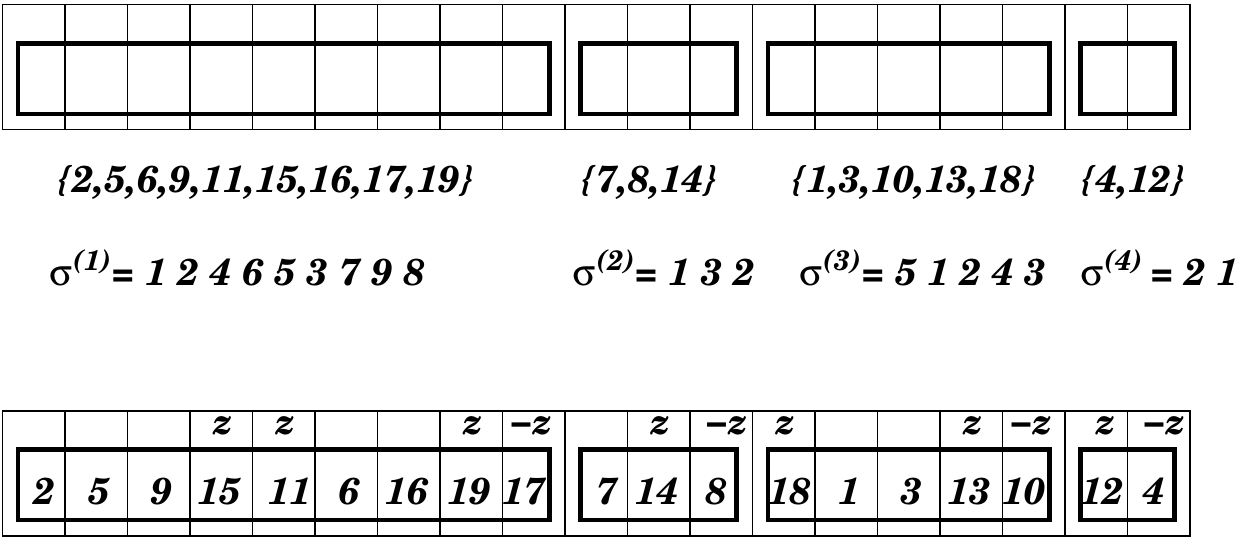}
    \caption{The construction of a filled, labeled brick tabloid.}
    \label{fig:CFLBTab}
  \end{center}
\end{figure}

It is easy to see that we can recover the triple $ \langle B, (S_1, \dots, S_{\ell(\la)}), (\sg^{(1)}, \dots,\sg^{(\ell(\la))}  ) \rangle$ from $B$ and the permutation $\sg$ which is obtained by reading the entries in the cells from right to left.  We let $\mathcal{O}_{\Gamma, n}$ denote the set of all filled labeled brick tabloids created this way. That is, $\mathcal{O}_{\Gamma, n}$ consists of all pairs $O = (B, \sg )$ where 
\begin{enumerate}
\item $B = (b_1, b_2, \dots, b_{\ell(\la)})$ is a brick tabloid of shape $n$,  
\item $\sg = \sg_1  \cdots \sg_n$ is a permutation in $S_n$ such that there is no $\Gamma$-match of $\sg$ which lies entirely in a single brick of $B$, and 
\item if there is a cell $c$ such that a brick $b_i$ contains both cells $c$ and $c+1$ and $\sg_c > \sg_{c+1}$, then cell $c$ is labeled with a $z$ and the last cell of any brick is labeled with $-z$.
\end{enumerate}

We define the sign of each $O$ to be $sgn(O) = (-1)^{\ell(\la)}.$ The weight $W(O)$ of $O$ is defined to be $q^{\inv(\sg)}$ times the product of all the labels $z$ used in the brick. Thus, the weight of the filled, labeled brick tabloid from Figure \ref{fig:CFLBTab} above is $W(O) = z^{11}q^{84}$.

It follows that 
\begin{equation}\label{eq:basic2}
[n]_q!\theta_{\Gamma}(h_n) = \sum_{O \in \mathcal{O}_{\Gamma,n}} sgn(O) W(O).
\end{equation}

Next we define a sign-reversing, weight-preserving mapping $J_{\Gamma}: \mathcal{O}_{\Gamma, n} \rightarrow \mathcal{O}_{\Gamma, n}$ as follows.  Let $(B,\sg) \in \mathcal{O}_{\Gamma, n}$ where $B=(b_1, \ldots, b_k)$ and $\sg = \sg_1 \ldots \sg_n$. Then for any $i$, we let $\mbox{first}(b_i)$ be the element in the left-most cell of $b_i$ and $\mbox{last}(b_i)$ be the element in the right-most cell of $b_i$. Then we read the cells of $(B,\sg)$ from left to right, looking for the first cell $c$ that belongs to either one of the following two cases. \\
\ \\
{\bf Case I.}  Either cell $c$ is in the first brick $b_1$ and is labeled with a $z$, or it is in some brick $b_j$, for $j > 1$, with either \begin{itemize} \item[i.] $\mbox{last}(b_{j-1}) < \mbox{first}(b_j)$ or  \item[ii.] $\mbox{last}(b_{j-1}) > \mbox{first}(b_j)$ and there is a $\tau$-match contained in the cells of $b_{j-1}$ and the cells $b_j$ that end weakly to the left of cell $c$ for some $\tau \in \Gamma$.\end{itemize} 
\ \\
{\bf Case II.} Cell $c$ is at the end of brick $b_i$ where $\sg_c > \sg_{c+1}$ and there is no $\Gamma$-match of $\sg$ that lies entirely in the cells of the bricks $b_i$ and $b_{i+1}$.\\

In Case I, we define $J_{\Gamma}((B,\sg))$ to be the filled, labeled brick tabloid obtained from $(B,\sg)$ by breaking the brick $b_j$ that contains cell $c$ into two bricks $b_j'$ and $b_j''$ where $b_j'$ contains the cells of $b_j$ up to and including the cell $c$ while $b_j''$ contains the remaining cells of $b_j$. In addition, we change the labeling of cell $c$ from $z$ to $-z$. In Case II, $J_{\Gamma}((B,\sg))$ is obtained by combining the two bricks $b_i$ and $b_{i+1}$ into a single brick $b$ and changing the label of cell $c$ from $-z$ to $z$. If neither case occurs, then we let $J_{\Gamma}((B,\sg)) = (B,\sg)$.

\begin{figure}[h]
  \begin{center}
   \includegraphics[width=0.7\textwidth]{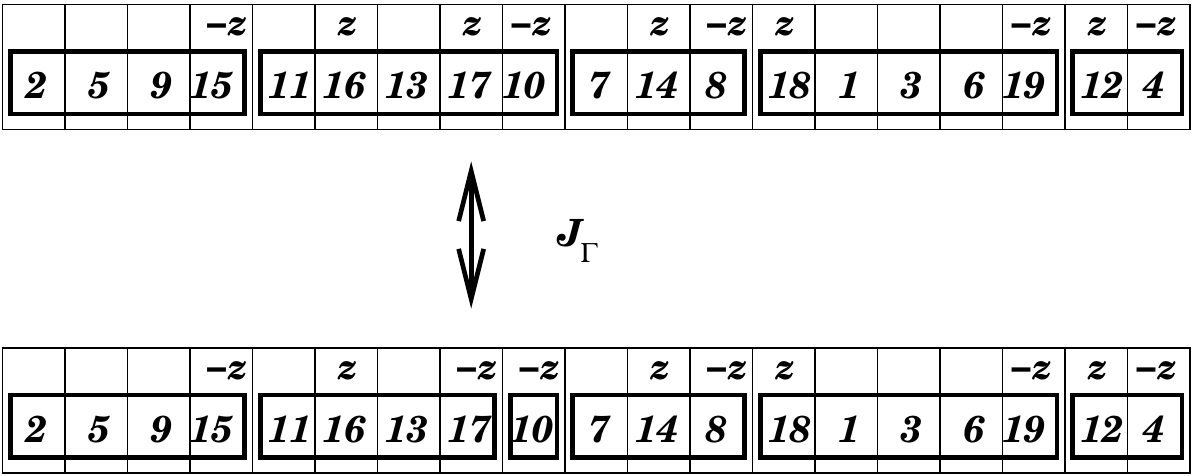}
   \caption{An example of the involution $J_{\Gamma}$.}
   \label{fig:JGamma}
  \end{center}
\end{figure}

For example, suppose $\Gamma = \{\tau\}$ where $\tau =14253$ and $(B,\sg) \in \mathcal{O}_{\Gamma,19}$ pictured at the top of Figure \ref{fig:JGamma}. We cannot use cell $c=4$ to define $J_{\Gamma}(B,\sg)$, because if we combined bricks $b_1$ and $b_2$, then $\red(9~15~11~16~13) = \tau$ would be a $\tau$-match contained in the resulting brick.  Similarly, we cannot use  cell $c=6$ to apply the involution because it fails to meet condition (b.2). In fact the first $c$ for which either Case I or Case II applies is cell $c=8$ so that $J_{\Gamma}(B,\sg)$ is equal to the $(B',\sg)$ pictured on the bottom of Figure \ref{fig:JGamma}.

We now prove that $J_{\Gamma}$ is an involution by showing $J_{\Gamma}^2$ is the identity mapping. Let $(B,\sg) \in \mathcal{O}_{\Gamma,n}$ where $B=(b_1, \ldots, b_k)$ and $\sg = \sg_1 \ldots \sg_n$. The key observation here is that applying the mapping $J_{\Gamma}$ to a brick in Case I will produce one in Case II, and vice versa. 

Suppose the filled, labeled brick tabloid $(B,\sg)$ is in Case I and its image $J_{\Gamma}((B,\sg))$ is obtained  by splitting some brick $b_j$ after cell $c$ into two bricks $b_j'$ and $b_j''.$  There are now two possibilities. \begin{itemize}

\item[(a.)] $c$ is in the first brick $b_1$. In this case, $c$ must be the first cell which is labeled with $z$ so that the elements in $b_1'$ will be increasing. Furthermore, since we are assuming there is no $\Gamma$-match in the cells of brick $b_1$ in $(B,\sg)$, there cannot be any $\Gamma$-match that involves the cells of bricks $b_1'$ and $b_1''$ in $J_{\Gamma}((B,\sg))$. Hence, when we consider $J_{\Gamma}((B,\sg))$, the first possible cell where we can apply $J_{\Gamma}$ will be cell $c$ because we can now combine the $b_1'$ and $b_1''$ .  Thus, applying $J_{\Gamma}$ to $J_{\Gamma}((B,\sg))$, we will recombine bricks $b_1'$ and $b_1''$ into $b_1$ and replace the label of $-z$ on cell $c$ by $z$. So $J_{\Gamma}(J_{\Gamma}((B,\sg))) =(B,\sg)$ in this case. 

\item[(b.)] $c$ is in brick $b_j$, where $j > 1$. Note that our definition of when a cell labeled $z$ can be used in Case I to define $J_{\Gamma}$ depends only on the cells and the brick structure to the left of that cell. Hence, we can not use any of the cells labeled $z$ to the left of $c$ to define $J_{\Gamma}(J_{\Gamma}((B,\sg)))$. Similarly, if we have two bricks $b_s$ and $b_{s+1}$ which lie entirely to the left of cell $c$ such that $\mbox{last}(b_s) = \sg_d > \mbox{first}(b_{s+1}) =\sg_{d+1}$, the criteria to use cell $d$ in the definition of $J_{\Gamma}$ on $J_{\Gamma}((B,\sg))$ depends only on the elements in bricks $b_s$ and $b_{s+1}$. Thus, the only cell $d$ which we could possibly use to define $J_{\Gamma}$ on $J_{\Gamma}((B,\sg))$ that lies to the left of $c$ is the last cell of $b_{j-1}$. However, our conditions that either $\mbox{last}(b_{j-1}) < \mbox{first}(b_{j}) = \mbox{first}(b_{j}')$ or  $\mbox{last}(b_{j-1}) > \mbox{first}(b_{j}) = \mbox{first}(b_{j}')$ with a $\Gamma$-match contained in the cells of $b_{j-1}$ and $b_j'$ force the first cell that can be used to define $J_{\Gamma}$ on $J_{\Gamma}((B,\sg))$ to be cell $c$. Thus, applying $J_{\Gamma}$ to $J_{\Gamma}((B,\sg))$, we will recombine bricks $b_j'$ and $b_j''$ into $b_j$ and replace the label of $-z$ on cell $c$ by $z$. So  $J_{\Gamma}(J_{\Gamma}((B,\sg))) =(B,\sg)$ in this case. 
\end{itemize}

Suppose $(B,\sg)$ is in Case II and we define $J_{\Gamma}((B,\sg))$ at cell $c$, where $c$ is last cell of $b_j$ and $\sg_c > \sg_{c+1}$. Then by the same arguments that we used in Case I, there can be no cell labeled $z$ to the left of this cell $c$ in $(B,\sg)$ nor $J(B,\sg)$. This follows from the fact that the brick structure before cell $c$ is unchanged between $(B,\sg)$ and $J(B,\sg)$. In addition, there can be no two bricks that lie entirely to the left of cell $c$ in $J_{\Gamma}((B,\sg))$ that can be combined under $J_{\Gamma}$. Thus, the first cell that we can use to define $J_{\Gamma}$ to $J_{\Gamma}((B,\sg))$ is cell $c$ and it is easy to check that it satisfies the conditions of Case I.  Thus, in going from $(B,\sg)$ to $J_{\Gamma}((B,\sg))$ we combine bricks $b_j$ and $b_{j+1}$ into a single brick $b$ and replaced the label on cell $c$ by $z$. Then it is easy to see that when applying $J_{\Gamma}$ to $J_{\Gamma}((B,\sg))$, we will split $b$ back into bricks $b_j$ and $b_{j+1}$ and change the label on cell $c$ back to $-z$. Thus  $J_{\Gamma}(J_{\Gamma}((B,\sg))) =(B,\sg)$ in this case.

Hence $J_{\Gamma}$ is an involution. It is clear that if $J_{\Gamma}(B,\sg) \neq (B,\sg)$,  then $$sgn(B,\sg)W(B,\sg) = -sgn(J_{\Gamma}(B,\sg))W(J_{\Gamma}(B,\sg)).$$ Thus, it follows from (\ref{eq:basic2}) that 
\begin{equation}\label{eq:basic3} 
n!\theta_\Gamma(h_n) =  \sum_{O \in \mathcal{O}_{\Gamma,n}} \sgn{O} W(O) = \sum_{O \in \mathcal{O}_{\Gamma,n}, J_{\Gamma}(O) =O} \sgn{O} W(O). 
\end{equation}
Hence if all permutations in $\Gamma$ start with 1, then 
\begin{equation}\label{eq:KEY}
U_{\Gamma,n}(y) = \sum_{O \in \mathcal{O}_{\Gamma,n}, I_{\Gamma}(O) =O} \sgn{O} W(O). 
\end{equation} 
Thus, to compute $U_{\Gamma,n}(y)$, we must analyze the fixed points of $J_{\Gamma}$.

Let $\Gamma$ be a finite set of permutations which all start with 1 and there is a $k \geq 2$ such that there exists a $\tau \in \Gamma$ with $\des(\tau) =k$ and for all $\alpha \in \Gamma$, $\des(\alpha) \leq k$. Let $\mathbb{Q}[y]$ be the set of rational functions in the variable $y$ over the rationals $\mathbb{Q}$ and let $\theta_\Gamma:\Lambda \rightarrow \mathbb{Q}[y]$ be the  ring homomorphism defined by setting  $\theta_{\Gamma}(e_0) =1$,  and  $\theta_{\Gamma}(e_n) = \frac{(-1)^n}{n!} \mbox{NM}_{\Gamma,n}(1,y)$ for $n \geq 1$. Then 
\begin{equation}\label{eq:keyGamma}
n!\theta_\Gamma (h_n) = \sum_{O \in \mathcal{O}_{\Gamma,n},J_\Gamma(O) = O}\sgn{O}W(O),
\end{equation}
where $\mathcal{O}_{\Gamma,n}$ is the set of objects and $J_\Gamma$ is the involution defined above. Suppose that $(B,\sg) \in \mathcal{O}_{\Gamma,n}$ where $B=(b_1, \ldots, b_k)$ and $\sg = \sg_1 \cdots \sg_n$. Then we have the following lemma describing the fixed points of the involution $J_{\Gamma}$.

\begin{lemma} \label{lem:keyGamma}
$(B,\sg)$ is a fixed point of $J_{\Gamma}$ if and only if it satisfies the following properties: 
\begin{enumerate}
\item if $i=1$ or $i > 1$ and $\mbox{last}(b_{i-1}) < \mbox{first}(b_i)$, then $b_i$ can have no cell labeled $z$ so that $\sg$ must be increasing in $b_i$, 

\item if $i > 1$ and $\sg_e = \mbox{last}(b_{i-1}) > \mbox{first}(b_i)=\sg_{e+1}$, then there must be a $\Gamma$-match contained in the cells of $b_{i-1}$ and $b_i$ which must necessarily involve $\sg_e$ and $\sg_{e+1}$ and there can be at most $k-1$ cells labeled $z$ in $b_i$, and

\item if $\Gamma$ has the property that, for all $\tau \in \Gamma$ such that $\des(\tau) = j \geq 1$, the bottom elements \footnote{If $\sg$ is a permutation with $\sg_i > \sg_{i+1}$, i.e. there is a descent in $\sg$ at position $i$, then we shall refer to $\sg_{i+1}$ as the bottom element of this descent.} of these descents are $2, \ldots, j+1$  when reading from left to right, then the first elements of each brick in $(B,\sg)$ form an increasing sequence.

\end{enumerate}
\end{lemma}

\emph{Proof.} Suppose $(B,\sg)$ is a fixed point of $J_{\Gamma}$. Then it must be the case that $(B,\sg)$ contains no cell $c$ that belongs to neither Case I nor Case II. That is, when applying the involution $J_{\Gamma}$ to $(B,\sg)$, we cannot break nor combine the bricks after any cell in $(B,\sg)$. 


For (1.), note that if there is a cell labeled $z$ in $b_i$ and  $c$ is the left-most cell of $b_i$ labeled with $z$, then $c$ satisfied the conditions of Case I. Thus there can be no cell labeled $z$ in $b_i$.

For (2.), note that if there is no $\Gamma$-match contained in the cells of $b_{i-1}$ and $b_i$, then $e$ satisfies the conditions of Case II.  Thus, there must be a $\Gamma$-match contained in the cells of $b_{i-1}$ and $b_i$. If there are $k$ or more cells labeled $z$ in $b_i$, then let $c$ be the $k$-th cell, reading from left to right, which is labeled with $z$. Then we know there is $\tau$-match contained in the cells of $b_{i-1}$ and $b_i$ which must necessarily involve $\sg_e$ and $\sg_{e+1}$ for some $\tau \in \Gamma$. But this $\tau$-match must end before cell $c$ since otherwise $\tau$ would have at least $k+1$ descents. Thus $c$ would satisfy the  the conditions of Case I.

To prove (3.), suppose there exist two consecutive bricks $b_i$ and $b_{i+1}$ 
with initial elements $\sg_e$ and $\sg_f$, respectively, such that $\sg_e > \sg_f$. There are two cases. \\
\ \\
{\bf Case A.}  $\sg$ is increasing in $b_i$. \\
Then $\sg_{f-1}$ is that last cell of $b_i$.  If $\sg_{f-1} < \sg_f$, then we know that $\sg_e \leq \sg_{f-1} < \sg_{f}$ which contradicts our choice of $\sg_e$ and $\sg_f$. Thus it must be the case that $\sg_{f-1} > \sg_f$.  But then there is $\tau \in \Gamma$ such that $\des(\tau) =j \geq 1$ and there is a $\tau$-match in the cells of $b_i$ and $b_{i+1}$ involving the $\sg_{f-1}$ and $\sg_f$. By our assumptions, $\sg_f$ can only play the role of $2$ in such a $\tau$-match. Hence there must be some $\sg_g$ with $e \leq g \leq f-2$ which plays the role of 1 in this $\tau$-match. But then we would have $\sg_e \leq \sg_g  < \sg_f$ which contradicts our choice of $\sg_e$ and $\sg_f$. Thus $\sg$ cannot be increasing in $b_i$. \\
\ \\
{\bf  Case B.} $\sg$ is not increasing in $b_i$. \\
In this case, by part (1), we know that it must be the case that $\sg_{e-1} = \mbox{last}(b_{i-1}) > \sg_e = \mbox{first}(b_i)$ and, by (2), there is $\tau \in \Gamma$ such that $\des(\tau) =j \geq 1$ and there is a $\tau$-match in the cells of $b_{i-1}$ and $b_{i}$ involving the cells $\sg_{e-1}$ and $\sg_e$. Call this $\tau$-match $\alpha$ and suppose that cell $h$ is the bottom element of the last descent in $\alpha$. It cannot be that $\sg_e =\sg_h$.  That is, there can be no cell labeled $z$ that occurs after cell $h$ in $b_i$ since otherwise the left-most such cell $c$ would satisfy the conditions of Case I of the definition of $J_{\Gamma}$. But this would mean that $\sg$ is increasing in $b_i$ starting at $\sg_h$ so that if $\sg_e =\sg_h$, then $\sg$ would be increasing in $b_i$ which contradicts our assumption in this case. Thus there is some $2 \leq i \leq j$ such that $\sg_{e}$ plays the role of $i$ in the $\tau$-match $\alpha$ and $\sg_h$ plays the role of $j+1$ in the $\tau$-match $\alpha$. But this means that $\sg_e$ is the smallest element in brick $b_i$. That is, let $\sg_c$ be the smallest element in  $b_i$. If $\sg_e \neq \sg_c$, then $\sg_c$ must be the bottom of some descent in $b_i$ which implies that $c \leq h$. But then $\sg_c$ is part of the $\tau$-match $\alpha$ which means that $\sg_c$ must be playing the role of one of $i+1, \ldots, j+1$ in the $\tau$-match $\alpha$ and $\sg_e$ is playing the role of $i$ in the $\tau$-match $\alpha$ which is impossible if $\sg_e \neq \sg_c$. It follows that $\sg_e \leq \sg_{f-1}$. Hence, it can not be that case that $\sg_{f-1} < \sg_f$ since otherwise $\sg_e < \sg_f$. Thus it must be the case that $\sg_{f-1} > \sg_f$. But this means that there exists some $\delta \in \Gamma$ such that $\des(\delta) =p \geq 1$ and there is a $\delta$-match in the cells of $b_i$ and $b_{i+1}$ involving the $\sg_{f-1}$ and $\sg_f$.  Call this $\delta$-match $\beta$. By assumption, the bottom elements of the descents in $\delta$ are $2,3, \ldots, p+1$ so that $\sg_f$ must be playing the role of $2,3, \ldots ,p+1$ in the $\delta$-match $\beta$. Let $\sg_g$ be the element that plays the role of $1$ in the $\delta$-match $\beta$. $\sg_g$ must be in $b_i$ since $\delta$ must start with 1. But then we would have that $\sg_e \leq \sg_g < \sg_f$ since $\sg_e$ is the smallest element in $b_i$. \\

Thus, both Case A and Case B are impossible and, hence the minimal elements in the bricks of $(B,\sg)$ must be increasing reading from left to right.  \qed

We note that if condition (3) of the Lemma fails, it may be that the first elements of the bricks do not form an increasing sequence. For  example, it is easy to check that if $\Gamma = \{15342\}$, then the $(B,\sg)$ pictured in Figure \ref{fig:Counter2} is such a fixed point of $J_{\Gamma}$.

\begin{figure}[htbp]
  \begin{center}
   \includegraphics{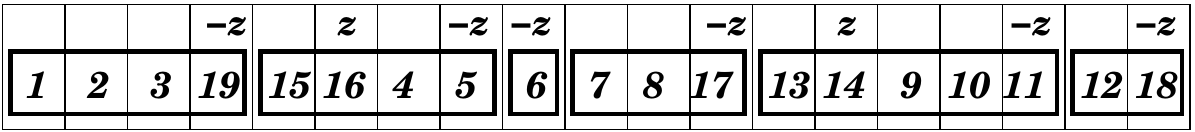}
   \caption{A fixed point of $J_{\{15342\}}$.}
   \label{fig:Counter2}
  \end{center}
\end{figure}

\section{The proof of Theorem \ref{thm:intro4}}

In this section, we shall prove Theorem \ref{thm:intro4}. To remind the readers of the result, we shall restate the theorem below. \begin{thm}
Suppose that $\alpha= \alpha_1 \ldots \alpha_j$ and $\beta = \beta_1 \ldots \beta_j$ are permutations in $S_j$ such that $\alpha_1 = \beta_1=1$, $\alpha_j = \beta_j$, $\des(\alpha) = \des(\beta)$, and $\alpha$ and $\beta$ have the minimal overlapping property.  Then 
\begin{equation}
\sum_{n \geq 0} \frac{t^n}{n!} \sum_{\sg \in \mathcal{NM}_n(\alpha)}  x^{\des(\sg)} = 
\sum_{n \geq 0} \frac{t^n}{n!} \sum_{\sg \in \mathcal{NM}_n(\beta)} x^{\des(\sg)}.
\end{equation} 
Thus $\alpha$ and $\beta$ are $\des$-c-Wilf equivalent.

If in addition, $\inv(\alpha) = \inv(\beta)$, then 
\begin{equation}
\sum_{n \geq 0} \frac{t^n}{[n]_q!} \sum_{\sg \in \mathcal{NM}_n(\alpha)}  
x^{\des(\sg)} q^{\inv(\sg)}= 
\sum_{n \geq 0} \frac{t^n}{[n]_q!} \sum_{\sg \in \mathcal{NM}_n(\beta)} 
x^{\des(\sg)} q^{\inv(\sg)}.
\end{equation} 
Thus $\alpha$ and $\beta$ are $(\des,\inv)$-c-Wilf equivalent.
\end{thm} This theorem is an immediate consequence of our next result.

\begin{theorem}\label{thm:KEY}
Let $\tau = \tau_1\tau_2\cdots \tau_p \in S_p$ be such that $\tau_1= 1, \tau_p = s$ where $2 \leq s < p$, and $\tau$ has the minimal overlapping property. Then $$\mbox{INM}_\tau(t,q,z)=\frac1{\mbox{IU}_\tau(t,q,z)} \text{ where } \mbox{IU}_\tau(t,q,z)=1+\sum_{n\geq1} \mbox{IU}_{\tau,n}(q,z)\frac{t^n}{[n]_q!},$$ with $\mbox{IU}_{\tau,1}(q,z)=-z$, and for $n \geq 2,$
\begin{align*}
\displaystyle \mbox{IU}_{\tau,n} (q,z) = (1-z)\mbox{IU}_{\tau, n-1}(q,z) - 
z^{\des(\tau)}q^{\inv(\tau)} \qbin{n-s}{p-s}{q}U_{\tau, n-p+1}(q,z).
\end{align*}
\end{theorem}

\emph{Proof.} When $n=1,$ the only fixed point comes from the configuration that consists of a single cell filled with 1 and labeled $-z$. Therefore, it must be the case that $\mbox{IU}_{\tau,1}(q,z)=-z.$  

For $n \geq 2,$ let $(B,\sg)$ be a fixed point of $J_{\tau}$ where $B=(b_1, \ldots, b_k)$ and $\sg=\sg_1 \cdots \sg_n$.  We claim that 1 is in the first cell of  $(B,\sg)$. To see this, suppose 1 is in cell $c$ where $c >1$. Hence $\sg_{c-1} > \sg_c$.  We claim that whenever $\sg_{c-1} > \sg_c$, $\sg_{c-1}$ and $\sg_c$ must be elements of some $\tau$-match in $\sg$.  That is, $c$ cannot be in brick $b_1$ because the elements in the first brick of any fixed point must be increasing. So we assume that $c$ is in brick $b_i$ where $2 \leq i \leq k$. If $c$ is the first cell of $b_i$, then $\mathrm{last}(b_{i-1}) > \mathrm{first}(b_i)$ and there must be a $\tau$-match in the cells of $b_{i-1}$ and $b_i$ which involves cells $c-1$ and $c$. If $c$ is not the first cell of $b_i$, then we can not have that $\mathrm{last}(b_{i-1}) < \mathrm{first}(b_i)$ since this would force $\sg$ to be increasing in the cells of $b_i$.  Thus, we must have that $\mathrm{last}(b_{i-1}) > \mathrm{first}(b_i)$ and there must be a $\tau$-match in the cells of $b_{i-1}$ and $b_i$. 
This $\tau$-match cannot end before cell $c$ since then $c$ would satisfy the conditions of Case I of our definition of $J_{\tau}$ which would contradict the fact that $(B,\sg)$ is a fixed point of $J_{\tau}$. Hence, cell $c$ must be part of this $\tau$-match. Thus if  $\sg_c=1$ where $c>1$, then $\sg_{c-1}$ and $\sg_c$ are elements of a $\tau$-match in $\sg$. But since $\tau$ starts with 1, the only role $\sg_c =1$ can play is a $\tau$-match is 1 and hence $\sg_{c-1}$ and $\sg_c$ cannot be elements of a $\tau$-match in $\sg$. Hence, $\sg_1 = 1.$ We now have two cases.

\ \\
{\bf Case 1.} There is no $\tau$-match in $(B,\sg)$ that starts from the first cell.\\
\ \\
In this case, we claim that 2 must be in cell 2 of $(B,\sg)$. By contradiction, suppose 2 is in cell $c$ where $c \ne 2$. For any $c > 2,$ it is easy to see that $\sg_{c-1} > 2 = \sg_c$ so there is a decrease between the two cells $c-1$ and $c$ in $(B,\sg)$. By our argument above, there must exist a $\tau$-match $\alpha$ that involves the two cells $c-1$ and $c$. In this case, $\alpha$ must include 1 which is in cell $1$ because it must be the case that 1 and 2 play the role of 1 and 2 in the $\tau$-match $\alpha$, respectively. This contradicts our assumption that there is no $\tau$-match starting from the first cell. Hence, $\sg_2 = 2$.

In this case there are two possibilities, namely, either (i) 1 and 2 are both in the first brick $b_1$ of $(B,\sg)$ or (ii) brick $b_1$ is a single cell filled with 1 and 2 is in the first cell of the second brick $b_2$ of $(B,\sg)$.  In either case, we know that 1 is not part of a $\tau$-match in $(B,\sg)$. So if we remove cell 1 from $(B,\sg)$ and subtract $1$ from the elements in the remaining cells, we will obtain a fixed point $(B'\sg')$ of $J_{\Gamma}$ in $\mathcal{O}_{\Gamma,n-1}.$ 

Moreover, we can create a fixed point $O=(B,\sg) \in \mathcal{O}_n$ satisfying the three conditions of Lemma \ref{lem:keyGamma} where $\sg_2 =2$ by starting with a fixed point $(B',\sg') \in \mathcal{O}_{\Gamma,n-1}$ of $J_\Gamma$, where $B' =(b_1', \ldots, b_r')$ and $\sg' =\sg_1' \cdots \sg_{n-1}'$, and then letting $\sg = 1 (\sg_1'+1) \cdots (\sg_{n-1}' +1)$, and setting  $B = (1,b_1', \ldots, b_r')$ or setting $B = (1+b_1', \ldots, b_r')$.

It follows that fixed points in Case 1 will contribute $(1-z)\mbox{IU}_{\Gamma,n-1}(q,z)$ to  $\mbox{IU}_{\Gamma,n}(q,z)$.

\ \\
{\bf Case 2.} There is a $\tau$-match in $(B,\sg)$ that starts from the first cell. \\
\ \\
In this case, the $\tau$-match that starts from the first cell of $(B,\sg)$ must involve the cells of the first two bricks $b_1$ and $b_2$ in $(B,\sg)$. Since there is no decrease within the first brick $b_1$ of $(B,\sg),$ it must be the case that the first brick $b_1$ has exactly $d$ cells, where $1< d <p$ is the position of the first descent in $\tau,$ and the brick $b_2$ has at least $p-d$ cells. Furthermore, we can see that the brick $b_2$ consists of exactly $\des(\tau)-1$ decreases, positioned according to their corresponding descents in $\tau$. We first claim that all the integers in $\{1,\ldots,s-1,s\}$ must belong to the first $p$ cells of $(B,\sg)$. To see this, suppose otherwise and let $m = \min\{ i: 1 \leq i \leq s, \sg_k = i \mbox{~for~some~} k > p \}  $. That is, $m$ is the smallest integer from $\{1,\ldots, s-1,s\}$ that occupies a cell $k$ strictly to the right of cell $p$ in $(B,\sg)$. It follows that $m$ is the smallest number that occupies a cell strictly to the right of cell $p$ in $O$ and thus, it is the case that $\sg_{k-1}\geq s > m = \sg_k$. Then there are three possibilities: 
\begin{itemize}
\item[(i)] brick $b_2$ has more than $p-d$ cells and $m$ is in brick $b_2$, 
\item[(ii)] $m$ starts some brick $b_j$ for $j > 2$, or 
\item[(iii)] $m$ is in the middle of some brick $b_j$ for $j > 2$.
\end{itemize} 
We will show that each of these cases contradicts our assumption $(B,\sg)$ is a fixed point of $J_{\Gamma}$.

In case (i), since $\sg_{k-1} > \sg_k$, there is a decrease in brick $b_2$ that occurs strictly to the right of cell $p.$ However, due to the $\tau$-match starting from cell 1 of $O,$ brick $b_2$ already has $\des(\tau)-1$ descents, the maximum number of allowed descents in a brick. Thus, by the second property of Lemma \ref{lem:keyGamma}, this is a contradiction. 

In case (ii), since $\mbox{last}(b_{j-1}) = \sg_{k-1} > \sg_k = \mbox{first}(b_j),$ by Lemma \ref{lem:keyGamma}, there must be a $\tau$-match that is contained in the cells of $b_{j-1}$ and $b_j$ and ends weakly to the left of cell $k$ which contains $m$. Since $\tau$ is a minimal overlapping permutation, the only possible $\tau$-match beside from the first one that starts from cell 1 in $(B,\sg)$ must occur weakly to the right of cell $p$ in $O.$ However, since $m$ is the smallest number in the cells to the right of cell $p$ and $\tau$ starts with 1, any match that involves $m$ must also start from this cell. Thus, we can never have a $\tau$-match in $(B,\sg)$ that involves both cells $k-1$ and $k$ in $(B,\sg)$.  

In case (iii), suppose that $m$ occupies cell $k$ that is in the middle of brick $b_j.$ There are now two possibilities between the last cell of $b_{j-1}$ and the first cell of brick $b_j$: either $\mbox{last}(b_{j-1}) < \mbox{first}(b_j)$ or $\mbox{last}(b_{j-1}) > \mbox{first}(b_j)$. If $\mbox{last}(b_{j-1}) < \mbox{first}(b_j)$ then we can simply break the brick $b_j$ after cell $k-1,$ contradicting the fact that $(B,\sg)$ is a fixed point. On the other hand, if $\mbox{last}(b_{j-1}) > \mbox{first}(b_j)$ then by Lemma \ref{lem:keyGamma}, there must be a $\tau$-match that ends weakly to the left of cell $k,$ and involves the two cells $k-1$ and $k.$ However, by previous argument, this cannot hold.

Hence, it must be the case that all the integers $\{1,2,\ldots, s-1,s\}$ belong to the first $p$ cells of $(B,\sg)$. Furthermore, we only have one way to arrange these entries, according to their respective position within the $\tau$-match. This also implies that $\sg_p = s$. We will then choose $p-s$ numbers and fill these numbers in the empty cells within the first $p$ cells of $(B,\sg)$ such that $\red(\sg_1\sg_2\cdots\sg_p) = \tau$. There are $\binom{n-s}{p-s}$ ways to do this and keeping track of the inversions between our choice of $p-s$ numbers and the elements of $(B,\sg)$ which occurs after cell $p$, we obtain a factor of $\qbin{n-s}{p-s}{q}$ from our possible choices.  Then we have to count the inversions among the first $p$ elements of $(B,\sg)$, which contributes a factor of $q^{\inv(\tau)}$. We notice that since $\tau$ has the minimal overlapping property, the next possible $\tau$-match in $(B,\sg)$ must start from cell $p$ that contains $s$. In addition, according to Lemma \ref{lem:keyGamma}, any brick in a fixed point of the involution can have at most $\des(\tau)-1$ descents within the brick so there cannot be any descents in $b_2$ after cell $p$.  By construction, $\sg_p =s$ is less than the elements which occur to the right of cell $p$. Therefore, we can remove the first $p-1$ cells of $(B,\sg)$ and obtain a fixed point $(B,\sg')$ of length $n-p+1.$

This process is also reversible. Suppose $\tau\in S_p$ is a minimal overlapping permutation with $\tau_1 = 1, \tau_p = s,$ and the first descent in $\tau$ occurs at position $d$. Given a fixed point $(B',\sg')$ of length $n-p+1$ where $B' = (b_1', \ldots,b_r')$ and a choice $T$ of $p-s$ elements from $\{s+1,\ldots, n\}$, we let $\sg^{*}$ be the permutation of $\{1, \ldots, s\} \cup T$ such that $\red(\sg^*) =\tau$ and $\sg^{**}$ be the permutation of $\{1, \ldots, n\}-(\{1, \ldots, s\} \cup T)$ such that $\red(\sg^*) = \sg'$. Then if we let $\sg = \sg^*\sg^{**}$ and $B =(d,p-d-1+b_1',b_2', \ldots, b_r')$, then $(B,\sg)$ will be a fixed point of $J_{\Gamma}$ of length $n$ that has $\tau$-match starting in cell 1. 

It follows that the contribution of the fixed points in Case 2 to $\mbox{IU}_{\tau,n} (q,z)$ is $$ - z^{\des(\tau)}q^{\inv(\tau)}
\qbin{n-s}{p-s}{q}\mbox{IU}_{\tau, n-p+1}(q,z).$$  

Combining Cases 1 and 2, we see that for $n \geq 2$,
\begin{align}\label{rec:Up}
\displaystyle \mbox{IU}_{\tau,n} (q,z) = (1-z)\mbox{IU}_{\tau, n-1}(q,z) - z^{\des(\tau)}q^{\inv(\tau)} \qbin{n-s}{p-s}{q}\mbox{IU}_{\tau, n-p+1}(q,z)
\end{align}
which is what we wanted to prove. \qed

It is easy to see that Theorem \ref{thm:intro4} follows immediately from Theorem \ref{thm:KEY}. That is, Theorem \ref{thm:KEY} shows that for a minimal overlapping permutation $\tau \in S_j$ that starts with 1, the generating function 
$$\mbox{INM}_{\tau}(t,1,z) = 1 + \sum_{n \geq 1} \frac{t^n}{n!} \sum_{\sg \in \mathcal{NM}_n(\tau)} z^{\des(\sg) +1}$$
depends only on $s = \tau_j$ and $\des(\tau)$.  Thus if $\alpha$ and $\beta$ are minimal overlapping permutations which start with 1 and end with $s$ and $\des(\alpha) = \des(\beta)$, then $\mbox{INM}_{\alpha}(t,1,z) = \mbox{INM}_{\beta}(t,1,z)$ so that $\alpha$ and $\beta$ are $\des$-c-Wilf equivalent. Similarly, Theorem \ref{thm:KEY} shows that for a minimal overlapping permutation $\tau \in S_j$ that starts with 1, the generating function 
$$\mbox{INM}_{\tau}(t,q,z) = 1 + \sum_{n \geq 1} \frac{t^n}{[n]_q!} \sum_{\sg \in \mathcal{NM}_n(\tau)} z^{\des(\sg) +1}q^{\inv(\sg)}$$ depends only on $s = \tau_j$, $\des(\tau)$, and $\inv(\tau)$.  Thus if $\alpha$ and $\beta$ are minimal overlapping permutations which start with 1 and end with $s$ and $\des(\alpha) = \des(\beta)$ and $\inv(\alpha) = \inv(\beta)$, then $\mbox{INM}_{\alpha}(t,q,z) = \mbox{INM}_{\beta}(t,q,z)$ so that $\alpha$ and $\beta$ are $(\des,\inv)$-c-Wilf equivalent.

There are lots of examples of minimal overlapping permutations $\alpha$ and $\beta$ for which the hypothesis of Theorem \ref{thm:intro4} apply. For example, consider $n =5$. Since we are only interested in permutations that start with 1, we know that such a permutation $\alpha$ starts with a rise. Then $\alpha$ cannot end in a rise since otherwise $\alpha$ is not minimal overlapping. Thus $\alpha$ must start with 1 and end in a descent. There are no such permutations that end in 5 and there are only two such permutations that end in 4, namely, $12354$ and $13254$ and these two permutations do not have the same number of descents. This leaves us 10 possible permutations to consider which we have listed in the following table. For each such $\sg$, we have list $\des(\sg)$, $\inv(\sg)$, and indicated whether is minimal overlapping.

\begin{center}
\begin{tabular}{|l|l|l|l|}
\hline
$\sg$ & $\des(\sg)$ & $\inv(\sg)$ & Is minimal overlapping? \\
\hline
12453 & 1 & 1 & yes \\
\hline
12543 & 2 & 3 & yes \\
\hline
14253 & 2 & 3 & no \\
\hline
15243 & 2 & 4 & no \\
\hline
13452 & 1 & 3 & yes \\
\hline
13542 & 2 & 4 & yes \\
\hline
14352 & 2 & 4 & yes \\
\hline
14532 & 2 & 5 & yes \\
\hline
15342 & 2 & 5 & yes \\
\hline
15432 & 3 & 6 & yes \\
\hline
\end{tabular}
\end{center}

Thus, Theorem \ref{thm:intro4} tells us that all the elements in the set $\{13542,14352,14532,15342\}$ are $\des$-c-Wilf equivalent. It also tells that the same set breaks up into 2 $(\des,\inv)$-c-Wilf equivalence classes, namely, $\{13542,14352\}$ and $\{14532,15342\}$.

Another natural question to ask is whether the size of $(\des,\inv)$-c-Wilf equivalence classes can get arbitrarily large as $n$ goes to infinity. The answer to this question is yes. First, it is easy to see that if $\sg$ is a permutation that starts with 1 and ends with 2, it is automatically minimal overlapping. That is, if $\sg = \sg_1 \ldots \sg_n$ where $\sg_1 =1$ and $\sg_n =2$, then there can be no $2 \leq i\leq n-1$ such that the first $i$ elements of $\sg$ has the same relative order as the last $i$ elements of $\sg$ because in the first $i$ elements of $\sg$ the smallest element is at the start while in the last $i$ elements of $\sg$, the smallest element is at the end. 

Now consider three consecutive elements $ x,x+1,x+2$. Then the sequences 
$t_1(x) =(x+1)(x+2)x$ and $t_2(x) =(x+2)x(x+1)$ each have one descent and two inversions. It follows that if we start with the permutation $\sg = 1~t_1(3)~t_1(6)~t_1(9)\cdots t_1(3n)~2$, then we can replace any of the sequence $t_1(3k)$ by its corresponding sequence $t_2(3k)$ and it will keep the inversion number and the descent number of the permutation the same.  Thus, the size of the $(\des,\inv)$-c-Wilf equivalence class of $\sg$ is at least $2^n$. 

There are lots of other examples of this type. For example, consider four consecutive elements $ x,x+1,x+2,x+3$.  Then the sequences $s_1(x) =(x+1)(x+2)x(x+3)$ and $s_2(x) =x(x+3)(x+1)(x+2)$ each have one descent and two inversions. It follows that if we start with the permutation $\tau = 1~s_1(3)~s_1(7)~s_1(11)\ldots s_1(4n-1)~2$, then we can replace any of the sequence $s_1(4k-1)$ by its corresponding sequence $s_2(4k-1)$ and it will keep the inversion number and the descent number of the permutation the same. This same argument can also be extended to permutations $\sg \in S_n$ that start with $123\cdots k$ and end with $k+1$, for any $k > 0.$ Hence, the size of the $(\des,\inv)$-c-Wilf equivalence class of $\tau$ is at least $2^n$.

\section{The proof of Theorem \ref{thm:intro5}}

In Theorem \ref{thm:intro5}, we study the $(\des,\LRmin)$-c-Wilf equivalent relation and its $q$-analog which arise as another consequence of Theorem \ref{thm:KEY}. First, we observe that for any permutation $\tau$, $\mbox{NM}_{\tau}(t,1,y) = \mbox{INM}_{\tau}(t,z,1)$ and hence, $U_{\tau}(t,y) = \mbox{IU}_{\tau}(t,z,1)$.  Thus, if $\alpha$ and $\beta$ are minimal overlapping permutations which start with 1 and end with $s$ with $\des(\alpha) = \des(\beta)$, then $U_{\alpha}(t,y)= U_{\beta}(t,y)$. This leads to $$\mbox{NM}_{\alpha}(t,x,y) = \left( \frac{1}{U_{\alpha}(t,y)}\right)^x = \left( \frac{1}{U_{\beta}(t,y)}\right)^x =\mbox{NM}_{\alpha}(t,x,y).$$
Hence, if $\alpha$ and $\beta$ are minimal overlapping permutations which start with 1 and end with $s$ and $\des(\alpha) = \des(\beta)$, then $\alpha$ and $\beta$ are $(\des,\LRmin)$-c-Wilf equivalent. In fact, by relaxing the condition that $\alpha$ and $\beta$ start with 1, we can generalize this result for pairs of permutations $\alpha$ and $\beta$ that satisfy the condition which we refer to as mutually minimal overlapping.

Before proceeding with the proof of Theorem \ref{thm:intro5}, we first recall the definition of mutually minimal overlapping permutations. Here, we say that $\alpha$ and $\beta$ are {\em mutually minimal overlapping} if $\alpha$ and $\beta$ are minimal overlapping and the smallest $n$ such that there exist a permutation $\sg \in S_n$ such that $\alpha\mbox{-}\mathrm{mch}(\sg) \geq 1$ and $\beta\mbox{-}\mathrm{mch}(\sg) \geq 1$ is $2j-1$. This ensures that in any permutation $\sg$, any pair of $\alpha$-matches, any pair of $\beta$ matches, and any pair of matches where one match is an $\alpha$-match and one match is a $\beta$-match can share at most one letter.

Note that if $\alpha = \alpha_1 \ldots \alpha_j$ and $\beta = \beta_1 \ldots \beta_j$ are minimal overlapping permutations in $S_j$ that start with 1 and end with 2, then $\alpha$ and $\beta$ are mutually minimal overlapping.  That is, it cannot be that there is $1 < i < j$ such that the last $i$ elements of $\alpha$ have the same relative order as the first $i$ elements of $\beta$ since the first $i$ elements of $\alpha$ has its smallest element at the start while the last $i$ elements of $\beta$ has it smallest element at the end. Similarly, it can not be that there is $1 < i < j$ such that that last the last $i$ elements of $\beta$ have the same relative order as the first $i$ elements of $\alpha$. On the other hand, if 
\begin{eqnarray*} 
\alpha &=& 1~9~3~8~2~7~6~5~4 \ \mbox{and} \\
\beta &=& 1~3~9~8~7~5~2~6~4, 
\end{eqnarray*}
then one can check that $\alpha$ and $\beta$ are minimal overlapping, $\des(\alpha) =\des(\beta) =4$, and $\inv(\alpha) = \inv(\beta) = 19$. However $\alpha$ and $\beta$ are not mutually minimal overlapping since the first 3 elements of $\alpha$ have the same relative order as that last three elements of $\beta$. 

We shall give a bijective proof for a slightly stronger version of Theorem \ref{thm:intro4}. In fact, Theorem \ref{thm:intro4} is the special case of the following result when $a=1$. 

\begin{thm}
Suppose $\alpha= \alpha_1 \ldots \alpha_j$ and $\beta = \beta_1 \ldots \beta_j$ are permutations in $S_j$ which are mutually minimal overlapping and there is an $1 \leq a < j$ such that $\alpha_i = \beta_i$ for $i \leq a$, $\alpha_a = \beta_a = 1$, $\alpha_j = \beta_j$, and $\des(\alpha) = \des(\beta)$. 

Then \begin{equation}\label{eq:abdes}
\sum_{n \geq 0} \frac{t^n}{n!} \sum_{\sg \in \mathcal{NM}_n(\alpha)}  x^{\des(\sg)} y^{\LRmin(\sg)} = 
\sum_{n \geq 0} \frac{t^n}{n!} \sum_{\sg \in \mathcal{NM}_n(\beta)} 
x^{\des(\sg)} y^{\LRmin(\sg)}.
\end{equation} 
Thus $\alpha$ and $\beta$ are $(\des,\LRmin)$-c-Wilf equivalent.

If in addition, $\inv(\alpha) = \inv(\beta)$, then 
\begin{equation}\label{eq:abdesq}
\sum_{n \geq 0} \frac{t^n}{[n]_q!} \sum_{\sg \in \mathcal{NM}_n(\alpha)}  x^{\des(\sg)} y^{\LRmin(\sg)} q^{\inv(\sg)}= 
\sum_{n \geq 0} \frac{t^n}{[n]_q!} \sum_{\sg \in \mathcal{NM}_n(\beta)} x^{\des(\sg)} y^{\LRmin(\sg)} q^{\inv(\sg)}.
\end{equation} 
Thus $\alpha$ and $\beta$ are $(\des,\LRmin,\inv)$-c-Wilf equivalent.
\end{thm}

\begin{proof}  For any $n \geq 0$, we can partition the elements of $S_n$ into four sets:
\begin{enumerate}

\item $A_n$ equals the set of $\sg \in S_n$ such that $\alpha\mbox{-}\mathrm{mch}(\sg) > 0$ and $\beta\mbox{-}\mathrm{mch}(\sg) = 0$, 

\item $B_n$ equals the set of $\sg \in S_n$ such that $\beta\mbox{-}\mathrm{mch}(\sg) > 0$ and $\alpha\mbox{-}\mathrm{mch}(\sg) = 0$, 

\item $C_n$ equals the set of $\sg \in S_n$ such that $\beta\mbox{-}\mathrm{mch}(\sg) > 0$ and $\alpha\mbox{-}\mathrm{mch}(\sg) > 0$,

\item $D_n$ equals the set of $\sg \in S_n$ such that $\beta\mbox{-}\mathrm{mch}(\sg) =  0$ and $\alpha\mbox{-}\mathrm{mch}(\sg) = 0$.
\end{enumerate}

Clearly $\mathcal{NM}_n(\alpha) = D_n \cup B_n$ and $\mathcal{NM}_n(\beta) = D_n \cup A_n$. Thus, to prove that 
$$\sum_{\sg \in \mathcal{NM}_n(\alpha)} z^{\des(\sg)} u^{\LRmin(\sg)} =\sum_{\sg \in \mathcal{NM}_n(\beta)} z^{\des(\sg)} u^{\LRmin(\sg)}, $$
we need only prove that 
$$\sum_{\sg \in A_n} z^{\des(\sg)} u^{\LRmin(\sg)} =\sum_{\sg \in B_n} z^{\des(\sg)} u^{\LRmin(\sg)}. $$

Thus, we need to define a bijection $\phi:A_n \rightarrow B_n$ such that for all $\sg \in A_n$, $\des(\sg) = \des(\phi(\sg))$ and $\LRmin(\sg) = \LRmin(\phi(\sg))$. One simply replaces each $\alpha$-match $\sg_i \ldots \sg_{i+j-1}$ in $\sg$ by the $\beta$-match where we rearrange $\sg_{i+1} \ldots \sg_{i+j-2}$ so that it matches $\beta$. Given our conditions on $\alpha$ and $\beta$, this mean that we will simply rearrange $\sigma_{i+a} \ldots \sg_{i+j-2}$ to match the order of the elements $\beta_{a+1} \ldots \beta_{j-1}$. Since $\alpha$ is minimal overlapping, the elements that we rearrange in any two $\alpha$ matches of $\sg$ are disjoint. Hence $\phi$ is well defined. 

The fact that $\alpha_a =\beta_a =1$ ensures that $\sg_{i+a-1}$ is less than each of the elements $\sg_{i+a} \ldots \sg_{i+j-2}$ so that rearranging these can not effect the number of left-to-right minima. So $\LRmin(\sg) = \LRmin(\phi(\sg))$. The fact that $\des(\alpha) = \des(\beta)$ ensures that our rearrangement  $\sg_{i+1} \ldots \sg_{i+j-1}$ does not effect the number of descents so that $\des(\sg) = \des(\phi(\sg))$. Moreover, if $\inv(\alpha) = \inv(\beta)$, then our rearrangement  $\sg_{i+a} \ldots \sg_{i+j-2}$ does not effect the number of inversions so that $\inv(\sg) = \inv(\phi(\sg))$.

Next we claim the fact that $\alpha$ and $\beta$ are mutually minimal overlapping ensures that $\phi(\sg)$ is in $B_n$.  That is, if $\phi(\sg)$ has an $\alpha$ match, then if must have been the case that there was $\alpha$-match $\sg_i \ldots \sg_{i+j-1}$ in $\sg$ such that the rearrangement of $\sg_{i+a} \ldots \sg_{i+j-2}$ or possibly two consecutive $\alpha$-matches in $\sg$ $\sg_i \ldots \sg_{i+2j-2}$ such that the rearrangement of $\sg_{i+a} \ldots \sg_{i+j-2}$ and the rearrangement of $\sg_{i+j-1+a} \ldots \sg_{i+2j-3}$ caused an $\alpha$-match to appear. In either case, this would mean that that there is an $\alpha$-match in $\phi(\sg)$ which shares more than 2 letters with a $\beta$-match in $\phi(\sg)$. This is impossible since $\alpha$ and $\beta$ are mutually minimal overlapping.

Finally, it is clear how to define $\phi^{-1}(\sg)$. One simply replaces each $\beta$-match $\sg_i \ldots \sg_{i+j-1}$ in $\sg$ by the $\alpha$-match where we rearrange $\sg_{i+a} \ldots \sg_{i+j-2}$ so that it matches $\alpha$. The same arguments will ensure that $\phi^{-1}$ is well defined and maps $B_n$ into $A_n$.  Thus $\phi$ proves theorem. 
\end{proof}

Finally, we observe that our proof of Theorem \ref{thm:KEY} can also be modified to prove the following theorem which allows us to study the c-Wilf equivalent relations between families of permutations.

\begin{theorem} \label{thm:set}
Suppose $\Gamma = \{\alpha^{(1)}, \ldots , \alpha^{(k)}\}$ is a set of  minimal overlapping permutations in $S_p$ which all start with 1 and $\alpha^{(i)}$ and $\alpha^{(j)}$ are mutually minimal overlapping for all $1 \leq i < j \leq k$. For each $1 \leq i \leq k$, let $s_i$ be the last element of $\alpha^{(i)}$. Then 
$$\mbox{INM}_{\Gamma}(t,q,z)=\frac1{\mbox{IU}_{\Gamma}(t,q,z)} \text{ where }
\mbox{IU}_{\Gamma}(t,y)=1+\sum_{n\geq1} \mbox{IU}_{{\Gamma},n}(q,z)\frac{t^n}{[n]_q!},$$ with $\mbox{IU}_{{\Gamma},1}(q,z)=-z$, and for $n \geq 2,$
\begin{align*}
\displaystyle \mbox{IU}_{{\Gamma},n} (q,z) = (1-z)\mbox{IU}_{{\Gamma}, n-1}(q,z) - \sum_{i=1}^k z^{\des(\alpha^{(i)})}q^{\inv(\alpha^{(i)})} \qbin{n-s_i}{p-s_i}{q}U_{\tau, n-p+1}(q,z).
\end{align*} 
\end{theorem}

\bibliographystyle{abbrvnat}
\bibliography{des-c-Wilf-bib}

\begin{thebibliography}{20}


\bibitem{AAM} R.E.L. Aldred, M.D. Atkinson, and D.J. McCaughan, 
Avoiding consecutive patterns in permutations, \emph{Adv. in Applied Math.},
{\bf 45:3} (2010), 449-461. 


\bibitem{BR} Q.T. Bach and J.B. Remmel, Generating functions for descents over permutations which avoid sets of consecutive patterns, \emph{Australas. J. Combin.}, {\bf 64} (2016), 194 - 231. 



\bibitem{BR2} Q.T. Bach and J.B. Remmel, Generating functions for descents over permutations which avoid consecutive patterns with multiple descents, in preparation.


\bibitem{B1} A.M. Baxter, Refining enumeration schemes to count according 
to inversion number, \emph{Pure Math. Appl.}, {\bf 21:2} (2010),
 136-160. 



\bibitem{B2} A.M. Baxter, Refining enumeration schemes to count according 
to permutation statistics, \emph{Electron. J. Comb.}, {\bf 21:2} (2014). 



\bibitem{BeckRem} D.A. Beck and J.B. Remmel, Permutation enumeration of the 
symmetric group and the combinatorics of symmetric functions, 
\emph{J. Combin. Theory Ser. A}, {\bf 72} (1995), 1-49. 


\bibitem{CK} A. Claesson and S. Kitaev, Classifiction of 
bijections between 321- and 132-avoiding permutations, 
\emph{S\'em. Lothar. Combin.}, {\bf B60d} (2008), 30pgs. 


\bibitem{DK} V. Dotsenko and A. Khoroshkin, Anick-type resolutions and 
consecutive pattern avoidance, arXiv:1002.2761v1 (2010).




\bibitem{DR} A. Duane and J. Remmel, Minimal overlapping patterns in colored
permutations, \emph{Electron. J. Combin.}, {\bf 18:2} (2011).


\bibitem{Eg1} O. E{\u{g}}ecio{\u{g}}lu and J. B. Remmel, Brick tabloids and the connection matrices between bases of symmetric functions, {\em Discrete Appl. Math.}, \textbf{34}, (1991), no.~1-3, 107--120, Combinatorics and theoretical computer science (Washington, DC, 1989).
  
\bibitem{EN}
S, Elizalde and M. Noy, Consecutive patterns in
permutations, \emph{Adv.
  in Appl. Math.}, \textbf{30} (2003), no.~1-2, 110-125, Formal power series and
  algebraic combinatorics (Scottsdale, AZ, 2001).

\bibitem{EN2} S. Elizalde and M. Noy, Clusters, generating 
functions and asymptotics for consecutive patterns in permutations,
\emph{Adv. in App. Math.}, {\bf 49} (2012), 351-374. 

 




\bibitem{EKP} R. Ehrenborg, S. Kitaev, and P. Perry, A spectral approach 
to consecutive pattern-avoiding permutations, \emph{J. Comb.}, 
{\bf 2} (2011), 305-353. 



\bibitem{GJ} I.P. Goulden and D.M. Jackson, \emph{Combinatorial Enumeration},
A Wiley-Interscience Series in Discrete Mathematics, John Wiley \& Sons Inc,
New York, (1983).











\bibitem{JR1} M. Jones and J. B. Remmel, Pattern Matching in the Cycle Structures of Permutations, {\em Pure Math. Appl.}, {\bf 22} (2011), 
173-208.



\bibitem{JR} M. Jones and J. B. Remmel, A reciprocity approach to computing 
generating functions for permutations with no pattern matches, 
\emph{Discrete Math. Theor. Comput. Sci. Proc.}, $23^{th}$ International Conference on Formal 
Power Series and Algebraic Combinatorics (FPSAC 2011), {\bf 119} (2011), 
551-562.


\bibitem{JR2} M. Jones and J. Remmel, A reciprocity method for 
computing generating function over the set of permutations 
with no consecutive occurrences of $\tau$, 
\emph{Discrete Math.}, {\bf 313} Issue 23 (2013), 2712-2729. 


\bibitem{JR3} M. Jones and J. Remmel, Generating functions 
for the number of permutations with no 
consecutive occurrences of  $1p23 \cdots (p-1)$ or $13\cdots(p-1)2p$, to appear in  
\emph{Pure Math. Appl.}.



\bibitem{Kit1}
S. Kitaev, Partially ordered generalized patterns,
\emph{Discrete Math.}, {\bf 298} (2005), 212-229.


\bibitem{Kitbook} S. Kiteav, {\em Patterns in permutations and words}, 
Springer-Verlag, 2011. 



\bibitem{MenRem}  A. Mendes and J.B. Remmel, 
Permutations and words counted by 
consecutive patterns, \emph{Adv. Appl. Math}, {\bf 37:4} (2006), 443-480.

\bibitem{oeis} N.~J.~A.~Sloane, The on-line encyclopedia of integer sequences,
published electronically at \phantom{*} {\tt
http:/$\!\!$/www.research.att.com/\~{}njas/sequences/}.


\bibitem{Stanley} R.P. Stanley, {\em Enumerative Combinatorics, vol. 2}, 
Cambridge Studies in Advanced Mathematics 62, Cambridge University 
Press, (1999). 






\end{thebibliography}
\label{sec:biblio}

\end{document}